\documentclass{svjour3}

\usepackage{amsmath}
\usepackage{amssymb}
\usepackage{graphicx}
\usepackage{arydshln}
\usepackage{hyperref}
\usepackage{enumitem}
\usepackage{comment}
\usepackage{algorithm,algorithmic}
\usepackage{dsfont}
\usepackage{tikz}

\newcommand{\E}{\mathbb{E}}
\newcommand{\vs}{\boldsymbol{s}}
\newcommand{\vx}{\boldsymbol{x}}

\DeclareMathOperator*{\flopt}{fl}

\definecolor{cerulean}{rgb}{0.16, 0.32, 0.75}

\newtheorem{mycorollary}[theorem]{Corollary}
\newtheorem{mylemma}[theorem]{Lemma}
\newtheorem{myremark}[theorem]{Remark}
\newtheorem{mydefinition}[theorem]{Definition}
\newtheorem{myexample}[theorem]{Example}

\numberwithin{theorem}{section}
\numberwithin{equation}{section}
\numberwithin{algorithm}{section}

\date{\today}

\title{Precision-aware
Deterministic and Probabilistic Error Bounds for 
Floating Point Summation\thanks{This research was supported in part by grants DMS-1745654 and DMS-1760374 from the National Science Foundation, and grant DE-SC0022085 from the Department of Energy.}}

\titlerunning{Error bounds for Floating Point Summation}

\author{Eric Hallman \and Ilse C.F. Ipsen}
\institute{Eric Hallman\at
Department of Mathematics, North Carolina State University, Raleigh, NC 27695-8205, USA\\\email{erhallma@ncsu.edu } \\
\texttt{https://erhallma.math.ncsu.edu/} 
 \and
Ilse C. F. Ipsen\at
Department of Mathematics, North Carolina State University, Raleigh, NC 27695-8205, USA \\
\email{ipsen@ncsu.edu}\\
 \texttt{https://ipsen.math.ncsu.edu/}
}

\begin{document}
\maketitle

\begin{abstract}
We analyze the forward error in the floating point summation of real numbers,
for computations in low precision or
extreme-scale problem dimensions that push
the limits of the precision.
We present a systematic recurrence for a martingale on a computational tree, 
which leads to explicit and interpretable bounds
without asymptotic big-O terms. Two probability parameters
strengthen the precision-awareness of our bounds:
one parameter controls the first order terms
in the summation error, while the second one is designed for controlling
 higher order terms in low precision or extreme-scale problem dimensions.
Our systematic approach yields
new deterministic and probabilistic error bounds for three classes
of mono-precision algorithms: general summation,
shifted general summation, and compensated (sequential) summation.
Extension of our systematic error analysis to mixed-precision summation
algorithms that allow any number of precisions yields the first
probabilistic bounds for the mixed-precision \texttt{FABsum} algorithm. 
Numerical experiments illustrate that the probabilistic bounds are accurate, and that  
among the three classes of mono-precision algorithms, compensated summation is generally the most accurate.
As for mixed precision algorithms, our recommendation is to minimize the magnitude
of intermediate partial sums relative to the precision in which they are computed.
\end{abstract}





\section{Introduction}
We analyze algorithms for the
summation $s_n=x_1+\cdots +x_n$ in floating point arithmetic of $n$ real numbers $x_1, \ldots, x_n$, and
bound the forward
error $e_n=\widehat{s}_n-s_n$ in the computed sum $\widehat{s}_n$ in terms of the unit roundoff~$u$.

Our bounds are designed for low precision computations, or 
extreme-scale problem dimensions $n$ that push the limits of the arithmetic precision with $n> u^{-1}$.
The idea is to set up a systematic recurrence for a martingale on a computational tree (Section~\ref{s_gprob2}),
and strengthen its precision-awareness
with the help of two probability parameters:
one to control the first order terms in the summation error; and a 
second one to control higher order terms which become more influential
with increasing problem dimension or decreasing precision. 
This precision-aware martingale makes possible a
unified and clean derivation of explicit
bounds, without asymptotic big-O terms,
for a wide variety of mono- and mixed-precison summation
algorithms. 

As an illustration, we derive new deterministic and
probabilistic bounds for
three classes of mono-precision algorithms: general summation on a computational tree (Section~\ref{s_general}), 
shifted general summation (Section~\ref{s_center}), 
and compensated summation (Section~\ref{s_compensated}). 
For compensated summation, our bounds imply that third and higher
order terms do not matter, unless the 
problem dimension $n$ is so extreme
as to have already
exceeded the limitations of the precision
with $n\gg u^{-2}$.

We extend our bounds to mixed-precision summation, 
allowing any number of precisions, on a computational tree (Section~\ref{s_mixedPrec}).
The special case of two precisions leads to the
first probabilistic bounds for the mixed-precision \texttt{FABsum} 
algorithm \cite{blanchard2020class}.
Numerical experiments
(Section~\ref{s_numex}) illustrate that the bounds are informative, and that,
among the three classes of mono-precision algorithms, compensated summation is the most accurate
method.

\subsection{Contributions}
We present systematic derivations for 
interpretable preci\-sion-aware forward error bounds for
summation
in mono- and mixed-precision  on a computational tree.

\paragraph{Martingales on a computational tree}
We present a systematic recurrence for martingales on a computational tree
(Theorem~\ref{thm:model2Theorem}, Corollary~\ref{c_28}), 
which makes possible a
unified and clean derivation of explicit
bounds, without asymptotic big-O terms,
for a wide variety of summation algorithms.

Our analysis of summation serves as a model problem  for 
systematic error analyses of higher level matrix computations
in mixed precision \cite{blanchard2020class}, or on
hardware with wider accumulators \cite{DemmelH03}.

\paragraph{Precision-aware bounds}
Our bounds are exact and hold to all orders. This is important when the problem dimension
exceeds the precision $n> u^{-1}$; or in low precision, where asymptotic terms $\mathcal{O}(u^2)$ in 
first-order bounds 
are too large to be ignored. Precision-awareness is strengthened
with two probability parameters:
one for controlling the first order terms in the summation error,
and a second one for controlling the $\mathcal{O}(u^2)$ terms.

\paragraph{General summation on a computational tree}
We extend the error bounds in \cite{higham2019new,ipsen2020probabilistic}
by customizing them to specific summation algorithms. Rather than 
depending on the number of inputs $n$, our bounds 
depend primarily on the height $h$ of the computational tree, which can be much smaller than $n$, particularly in parallel computations.

We derive a deterministic bound for the summation error $e_n$ that is proportional to $h\,u$ (Theorem~\ref{t_gdet})
and a probabilistic bound that is proportional to $\sqrt{h}\, u$.
The probabilistic bound treats the roundoffs as zero-mean random variables
that are mean-independent  
(Theorem~\ref{thm:model2Analysis}, Corollary~\ref{c_210}) 
and employs a novel
staggered martingale approach in the proof.

\paragraph{Shifted summation algorithms}
We extend the shifted \textit{sequential} summation 
in~\cite{blanchard2020class}
to shifted 
\textit{general} summation (Algorithm \ref{alg:shiftSum}). We derive
probabilistic bounds for mean-independent roundoffs (Theorem~\ref{thm:shifted}).

\paragraph{Compensated summation}
We derive a recursive expression for the exact error (Theorem~\ref{thm:comp_recurrences}), an explicit expression for the second-order error (Corollary~\ref{c_cs5}), and a probabilistic bound 
(Theorem~\ref{thm:comp_prob_err}) based on our martingale approach.
In particular (Remark~\ref{r_cfirst})
we note the discrepancy by a unit roundoff $u$ of existing bounds with ours,
\begin{equation*}
\widehat{s}_n= \sum_{k=1}^n(1+\rho_k)x_k, \qquad |\rho_k|\leq  3u+\mathcal{O}(nu^2).
\end{equation*}

\paragraph{Mixed precision summation}
We present bounds for mixed-precision summation, in any number of precisions,
on a computational tree (Theorem~\ref{thm:mixPrecError}). The special
case of two precisions yields the first probabilistic bounds 
(Corollary~\ref{c_fabsum}) for the
mixed-precision \texttt{FABsum} algorithm \cite{blanchard2020class}. 
\paragraph{Recommendation} 
For mono- and mixed precision algorithms,
pairwise summation based on a balanced binary tree
is the most accurate. Furthermore
(Remark~\ref{r_51}),
mixed-precision summation should try to minimize the magnitude of the intermediate
partial sums $s_k$ relative to the precision $u_k$ in which they are computed, 
that is, try to minimize $|u_ks_k|$ for all $k$.

Table \ref{tab:my_label} summarizes our contributions compared to recent related papers. 


\begin{table}
    \centering
    \begin{tabular}{l|c|c|c|c}
         & All Orders & Partial Sums & Martingale & Tree \\ \hline
       Higham/Mary \cite{higham2019new} & \checkmark & & & \\ \hline
        Ipsen/Zhou \cite{ipsen2020probabilistic} &\checkmark &&&  \\ \hline
       Higham/Mary \cite{higham2020sharper} &&\checkmark&\checkmark&  \\ \hline
        Connolly/Higham/Mary \cite{connolly2021stochastic} &\checkmark&&\checkmark&  \\ \hline
        This paper & \checkmark &\checkmark &\checkmark &\checkmark   \\ 
    \end{tabular}
    \caption{A summary of important features in probabilistic error bounds for summation.\\
Check marks in the four columns highlight the presence of the following features: the bounds hold indeed to all orders (`All Orders');
the bounds are expressed in terms of partial sums $s_k$, thus are tighter
than if they had been expressed in terms of inputs $x_k$ (`Partial Sums'); 
the bounds assume mean-independence of roundoffs rather than the stricter
notion of total independence ('Martingale'); the bounds apply 
to algorithms on any computational tree rather than just sequential summation (`Tree').}
    \label{tab:my_label}
\end{table}

\subsection{Modeling roundoff}
\label{s_model}
We assume the inputs $x_k$ are floating point numbers, that is, they can be stored exactly without error; and that the summation
produces no overflow or underflow.
Let $0<u<1$ denote the unit roundoff to nearest.

\paragraph{Individual roundoffs} 
Apply an operation
$\mathrm{op}\in\{+,-,*,/ \}$
to floating point numbers $x$ and $y$.
In the absence of underflow or
overflow, IEEE floating-point arithmetic
can be interpreted as computing
\cite{higham2002accuracy}
\begin{equation}\label{model:classical}
	\flopt(x\ \mathrm{op}\ y) = (x\ \mathrm{op}\  y)(1 + \delta_{xy}), \qquad |\delta_{xy}| \leq u.
\end{equation}

Our probabilistic bounds treat roundoffs as zero-mean mean-independent random variables.

\paragraph{Probabilistic model for sequences of roundoffs}
Assume the summation generates rounding errors $\delta_1,\delta_2,\ldots$, 
labeled in a linear
order consistent with the partial order of the underlying algorithm. 
We treat the $\delta_k$ as zero-mean
random variables 
that are mean independent\footnote{For simplicity, the conditioning also includes also those 
$\delta_{\ell}$, $1\leq \ell\leq k-1$, that are not descendants in the partial order. With stochastic rounding such $\delta_\ell$ would be fully independent
from $\delta_k$. } 
\begin{align}\label{model:second}
\E(\delta_k|\delta_1,\ldots,\delta_{k-1}) = \E(\delta_k)  = 0.
\end{align}
Mean-independence  (\ref{model:second}) is a weaker assumption than mutual independence of errors  but stronger than uncorrelated errors \cite{higham2020sharper}.
At least one mode of stochastic rounding
\cite{connolly2021stochastic} 
produces the mean-independent errors in (\ref{model:second}), but the 
stochastic rounding error bound
$|\delta_{xy}|\leq 2u$ is weaker 
than \eqref{model:classical}. 

\subsection{Probability theory}
For the derivation of the probabilistic bounds,  
we need a martingale, and a concentration inequality.

\begin{mydefinition}[Martingale \cite{mitzenmacher2005probability}] A sequence of random variables $Z_1,\ldots, Z_n$ is a martingale with respect to the sequence $X_1,\ldots,X_n$ if, for all $k\geq 1$,
	\begin{itemize}
		\item $Z_k$ is a function of $X_1,\ldots,X_k$, 
		\item $\E[|Z_k|]<\infty$, and 
		\item $\E\left[Z_{k+1}| X_1,\ldots, X_{k}\right] = Z_{k}$. 
	\end{itemize}
\end{mydefinition}

\begin{mylemma}[Azuma-Hoeffding inequality \cite{roch2015modern}]\label{l_azuma}
Let $Z_1,\ldots,Z_n$ be a martingale with respect to a sequence $X_1,\ldots, X_n$, and let $c_k$ be constants with
	\[
		 |Z_{k} - Z_{k-1}| \leq c_k, \qquad 2\leq k\leq n.
	\]
	Then for any $0<\delta<1$, with probability at least $1-\delta$,
	\begin{equation}\label{eqn:Azuma}
		|Z_n-Z_1| \leq  \left(\sum_{k=2}^nc_k^2\right)^{1/2}\sqrt{2\ln(2/\delta)}.
	\end{equation}
	\label{lemma:Azuma}
\end{mylemma}

If one or more of bounds $|Z_k-Z_{k-1}|\leq c_k$ are permitted to fail with 
probability at most $\eta$,
then a similar but weaker version 
of the Azuma-Hoeffding inequality still holds.

\begin{mylemma}[Relaxed Azuma-Hoeffding inequality \cite{chung2006concentration}]\label{l_azuma1}
Let $0<\eta<1$; $0<\delta<1$;
and $Z_1,\ldots,Z_n$ be a martingale with respect to a sequence $X_1,\ldots, X_n$.
Let $c_k$ be constants 
so that all bounds
	\[
		 |Z_{k} - Z_{k-1}| \leq c_k, \qquad 2\leq k\leq n.
	\]
	hold simultaneously with
	probability at least $1-\eta$.
	Then with probability at least $1-(\delta+\eta)$,
	\begin{equation}\label{eqn:Azuma1}
		|Z_n-Z_1| \leq  \left(\sum_{k=2}^nc_k^2\right)^{1/2}\sqrt{2\ln(2/\delta)}.
	\end{equation}
\end{mylemma}

\section{General summation on a computational tree}\label{s_general}
We present the algorithm for general summation 
(Algorithm \ref{alg:sum});
define its computational tree
(Definition~\ref{d:tree}); and derive error expressions and a deterministic error bound  (Section~\ref{s_gexact}); and finally
set up a martingale on a computational tree (Section~\ref{s_gprob2}).

\begin{algorithm}
    \centering
    \caption{General summation \cite[Algorithm~4.1]{higham2002accuracy}} \label{alg:sum}
    \begin{algorithmic}[1]
		\REQUIRE{A set of floating point numbers $\mathcal{S} = \{x_1,\ldots,x_n\}$}
		\ENSURE{$s_n = \sum_{k=1}^n{x_k}$}
		\FOR{$k = 2:n$}
		\STATE{Remove two elements $x$ and $y$ from $\mathcal{S}$}
		\STATE{$s_k = x + y$} \label{line:summation}
		\STATE{Add $s_k$ to $\mathcal{S}$}
		\ENDFOR
    \end{algorithmic}
\end{algorithm}

Denote by $s_k=\sum_{j=1}^k{x_j}$ the exact partial sum, by $\widehat{s}_k$ the sum computed in floating point arithmetic,
and by $e_k=\widehat{s}_k-s_k$ the 
absolute forward error, $2\leq k\leq n$.

\begin{mydefinition}[Computational tree for Algorithm~\ref{alg:sum}]\label{d:tree}
The partial order of pairwise summations in Algorithm~\ref{alg:sum} is represented
by
a binary tree with $2n-1$ vertices:
$n-1$ pairwise sums $s_2,\ldots, s_n$
to sum $n$ inputs $x_1,\ldots, x_n$.
Specifically,
\begin{itemize}
    \item Each vertex represents a pairwise sum $s_k$ or an input $x_k$.
    \item The root is the final sum $s_n$, and the leaves are the inputs $x_1,\ldots,x_n$. 
    \item Each pairwise sum $s_k = x+y$
    is a vertex with downward edges $(s_k,x)$ and $(s_k,y)$. Vertices $x$ and $y$ are the children of $s_k$.
    
\end{itemize}
The tree defines a partial ordering. 
    We say $j\prec k$ if $s_j$ is a descendant of $s_k$,
    and $j\preceq k$ if $s_j=s_k$ is possible. 
    \begin{itemize}
    \item The height of a node is the length of the longest downward path from that node to a leaf. 
\item Leaves have height zero. 
    \item The height of the tree is the height of its root. 
    Sequential summation yields a tree of height $n-1$.
\end{itemize}
\end{mydefinition}
Algorithm~\ref{alg:sum} imposes a {\em topological ordering} on the graph: $j\prec k$ implies that $j< k$. Thus if the nodes are visited in the order $s_2,\ldots,s_n$, no node is visited before its children.  Figure \ref{fig:DAG} shows two 
computational trees, one 
of height $n-1$ for sequential summation; and another of
height $\lceil \log_2 n\rceil$
for pairwise summation.

    \begin{figure}
       \centering
       \begin{minipage}{0.45\textwidth}
         \centering
        		\begin{tikzpicture}[scale=1.2]
		\tikzstyle{every node}+=[inner sep=0pt]
		\fill (0,0) circle (0.08) node [below=5] {$x_1$}; 
		\fill (1,0) circle (0.08) node [below=5] {$x_2$};  
		\fill (1.5,0.75) circle (0.08) node [below=5] {$x_3$}; 
		\fill (2,1.5) circle (0.08) node [below=5] {$x_4$}; 
		\draw (0,0) -- (1.5,2.25) -- (2,1.5); 
		\draw (0.5,0.75)--(1,0); 
		\draw (1.5,0.75) -- (1,1.5); 
	    \fill (0.5,0.75) circle (0.08) node [above left=3] {$s_2$};
		\fill (1,1.5) circle (0.08) node [above left =3] {$s_3$};
		\fill (1.5,2.25) circle (0.08) node [above left=3] {$s_4$};
		\end{tikzpicture}

      \end{minipage}\hfill
      \begin{minipage}{0.45\textwidth}
          \centering
        \begin{tikzpicture}[scale=1.2]
		\tikzstyle{every node}+=[inner sep=0pt]
		\fill (0,0) circle (0.08) node [below=5] {$x_1$}; 
		\fill (1,0) circle (0.08) node [below=5] {$x_2$};  
		\fill (2,0) circle (0.08) node [below=5] {$x_3$}; 
		\fill (3,0) circle (0.08) node [below=5] {$x_4$}; 
		\draw (0,0)-- (0.5,0.75) -- (1,0); 
		\draw (2,0) -- (2.5,0.75) -- (3,0); 
		\draw (0.5,0.75) -- (1.5,1.5) -- (2.5,0.75); 
		\fill (0.5,0.75) circle (0.08) node [above=5] {$s_2$};
		\fill (2.5,0.75) circle (0.08) node [above=5] {$s_3$};
		\fill (1.5,1.5) circle (0.08) node [above=5] {$s_4$};
		\end{tikzpicture}
        \end{minipage}
        \caption{Computational trees for two different summation orderings in
        Algorithm~\ref{alg:sum} for $n=4$.
        Left: sequential (a.k.a.~recursive) summation. Right: pairwise summation.}
        \label{fig:DAG}
    \end{figure}
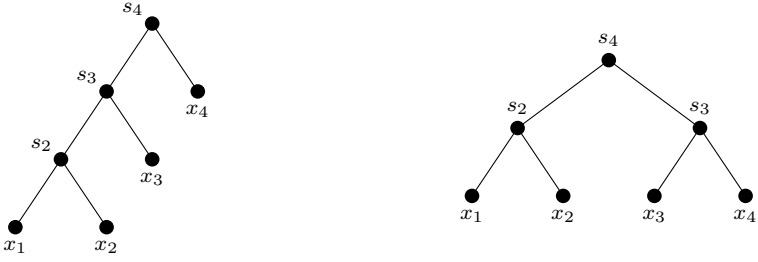

To make our bounds as tight as possible, we express them in terms of partial sums.
However, the dependence on the height of the computational tree is more
explicit when the bounds are expressed in terms of the inputs.
Below is the translation from partial sums to inputs.

\begin{mylemma}[Relation between partial sums and inputs]\label{l_rel}
If $h$ is the height of the computational tree in Algorithm \ref{alg:sum}, then
\begin{align*}
\sum_{k=2}^n{|s_k|}\leq h\,\sum_{j=1}^n{|x_j|},\qquad 
\sqrt{\sum_{k=2}^n{s_k^2}}\leq 
\sqrt{h}\,\sum_{j=1}^n{|x_j|}.
\end{align*}
\end{mylemma}
\begin{proof}
The first bound follows from the triangle inequality:
\[
    \sum_{k=2}^n|s_k| \leq \sum_{k=2}^n \sum_{j\prec k}|x_j| \leq \sum_{j=1}^n\sum_{j\prec k \preceq n}|x_j| \leq h\sum_{j=1}^n|x_j|,
\]
where in this context $j\prec k$ denotes the set of all leaves $x_j$ that are descendants of node $k$. 
The second bound follows from the first:  
\[
 \sum_{k=2}^n s_k^2 \leq \max_{2\leq j \leq n}|s_j|  \sum_{k=2}^n|s_k| \leq \left(\sum_{j=1}^n|x_j|\right)\left(h\sum_{j=1}^n|x_j|\right) = h\left(\sum_{j=1}^n|x_j|\right)^2. 
\]
\end{proof}

\subsection{Explicit expressions and deterministic bounds for errors on computational trees}\label{s_gexact}
We present two expressions for the error in Algorithm~\ref{alg:sum}
(Lemmas \ref{lemma:forwardErrorGenl} and \ref{lemma:forwardErrorRec}),
and a deterministic bound (Theorem~\ref{t_gdet}).

We generalize the error
for sequential summation in 
\cite[Lemma 3.1]{hallman2021refined}
to errors on computational trees.

\begin{mylemma}[First explicit expression]\label{lemma:forwardErrorGenl}
The error in Algorithm \ref{alg:sum} equals
	\begin{equation}\label{eqn:forwardErrorGenl}
		e_n = \widehat{s}_n - s_n = \sum_{k=2}^ns_k\delta_k\prod_{k\prec j \preceq n}(1+\delta_j). 
	\end{equation}
\end{mylemma}

Lemma~\ref{lemma:forwardErrorGenl} 
represents the forward error as a sum of local errors at a node, each perturbed by subsequent rounding errors. Truncating \eqref{eqn:forwardErrorGenl} yields the first order bound
\begin{equation}\label{e_gdet1}
	e_n = \sum_{k=2}^ns_k\delta_k + \mathcal{O}(u^2),
\end{equation}
which extends the result for sequential summation \cite[Lemma 2.1]{higham2020sharper}. Lemma~\ref{lemma:forwardErrorGenl} also allows us to conveniently obtain a deterministic error bound. 

\begin{theorem}\label{t_gdet}
If $h$ is the height of the computational tree for Algorithm~\ref{alg:sum} and
$\lambda_h\equiv (1+u)^h$,
then the error in Algorithm~\ref{alg:sum} is bounded by
\begin{align*}
|e_n| &\leq \sum_{k=2}^n|s_k||\delta_k|\prod_{k\prec j \preceq n}|1+\delta_j| 
\leq \lambda_h\,u\,\sum_{k=2}^n|s_k|\\
&\leq \lambda_h\,h\,u\,\sum_{j=1}^n{|x_j|}.
\end{align*}
\end{theorem}

\begin{proof}
The first bound is a direct consequence of Lemma~\ref{lemma:forwardErrorGenl}, while the last bound follows from Lemma~\ref{l_rel}. 
\end{proof}

\begin{myremark}\label{r_11}
A bound \cite[(4.3)]{higham2002accuracy} similar to the first one in Theorem~\ref{t_gdet},
\[|e_n|\leq u\sum_{k=2}^n|\widehat{s}_k|,\] 
is accompanied by the following observation:
\begin{quote}
\textit{In designing or choosing a summation method to achieve high accuracy, the aim should be to minimize the absolute values of the intermediate sums $s_k$.}
\end{quote}
Reducing the height of the computational tree often helps in this regard. The dependence on the height $h$ is explicitly visible in the second bound of Theorem~\ref{t_gdet}.
\end{myremark}

Since the sum in Lemma \ref{lemma:forwardErrorGenl} is not a martingale with respect to the errors $\delta_2,\ldots,\delta_n$, we present
an alternative geared towards the error model~(\ref{model:second}):
The sum in Lemma~\ref{lemma:forwardErrorRec} is a martingale if summed
in the original order, as shown in Section~\ref{s_gprob2}. 
Lemma~\ref{lemma:forwardErrorRec} 
also expresses the error in terms of exact partial sums, thereby making it more 
amenable to a probabilistic analysis
than the computed partial sums in 
$e_n = \sum_{k=2}^n\widehat{s}_k\tilde{\delta}_k$ \cite[(4.2)]{higham2002accuracy}.

\begin{mylemma}[Second explicit expression]\label{lemma:forwardErrorRec}
The error in Algorithm \ref{alg:sum} equals
	\begin{equation}\label{eqn:forwardErrorRec}
		e_n = \widehat{s}_n - s_n = \sum_{j=2}^n{(s_j+f_j)\delta_j}, 
	\end{equation}
where $f_j=0$ for all nodes whose children
are leaves, that is, represent a sum of two inputs $x_i$ and $x_j$. For all other nodes,
the {\em child-errors} satisfy the recurrence
	\begin{equation}\label{eqn:frecurrence}
f_k\equiv\sum_{j \prec k}(s_j + f_j)\delta_j.
	\end{equation}
\end{mylemma}

\begin{proof} Express the computed parent sum in 
line \ref{line:summation} of Algorithm \ref{alg:sum} 
as the sum of the computed children $\widehat{x}=x+e_x$ and $\widehat{y}=y+e_y$,  
\begin{equation*}
    \widehat{s}_k = (\widehat{x} + \widehat{y})(1+\delta_k), \qquad 2\leq k\leq n,
\end{equation*}
where $e_x=e_y=0$ if
$x$ and $y$ are inputs $x_i$ and $x_j$.
Highlight the error in the computed children,
\begin{align*}
s_k + e_k&=\widehat{s}_k= ((x + e_x) + (y+e_y))(1+\delta_k)
	=(s_k+e_x+e_y)(1+\delta_k)\\
&=	\underbrace{(e_x + e_y)}_{f_k}(1+\delta_k)
+s_k\delta_k+s_k
	\end{align*}
	to obtain the error in the computed parent 
\begin{equation*}
  e_k = f_k + (s_k + f_k)\delta_k,\qquad 
    2\leq k\leq n.
\end{equation*}
Now unravel the recurrence for $f_k$,
where $f_j=0$ for all nodes $j$ with two
leaf children.
\end{proof}

We refer to the terms $f_k$ as {\em child-errors}, since at any given node, $f_k$ is equal to the sum of the errors in the computed children.

\begin{myexample}
A pairwise tree summation for $n=8$ illustrates the recurrences for the child-errors in Lemma~\ref{lemma:forwardErrorRec}.
\begin{enumerate}
\item Sums of leaf nodes:
The exact sums are
\begin{align*}
s_2=x_1+x_2,\quad s_3=x_3+x_4,\quad s_4=x_5+x_6,\quad s_5=x_7+x_8,
\end{align*}
while the computed sums are
$\hat{s}_j=s_j+s_j\delta_j$ with child-errors $f_j=0$ for $2\leq j\leq 5$.
\item Second level:
The exact sums are
$s_6=s_2+s_3$ and $s_7=s_4+s_5$ while the computed sums are
\begin{align*}
\hat{s}_6&=(\hat{s}_2+\hat{s}_3)(1+\delta_6)
= \underbrace{(s_2\delta_2+s_3\delta_3)}_{f_6}(1+\delta_6)+s_6\delta_6+s_6\\
 &= f_6+ (s_6+f_6)\delta_6 +s_6\\
 \hat{s}_7&=(\hat{s}_4+\hat{s}_5)(1+\delta_7)
= \underbrace{(s_4\delta_4+s_5\delta_5)}_{f_7}(1+\delta_7)+s_7\delta_7+s_7\\
& = f_7+ (s_7+f_7)\delta_7 +s_7.
\end{align*}
With $f_j=0$, $2\leq j\leq 5$, the child-errors are 
\begin{align*}
f_6&=s_2\delta_2+s_3\delta_3=(s_2+f_2)\delta_2+(s_3+f_3)\delta_3=
\sum_{j \prec 6}{(s_j + f_j)\delta_j}\\
f_7&=s_4\delta_4+s_5\delta_5=(s_4+f_4)\delta_4+(s_5+f_5)\delta_5=
\sum_{j \prec 7}{(s_j + f_j)\delta_j}.
\end{align*}
\item Final level:
The exact sum is
$s_8=s_6+s_7$ while the computed sum is
\begin{align*}
\hat{s}_8&=(\hat{s}_6+\hat{s}_7)(1+\delta_8)\\
&= \underbrace{\left(f_6+(s_6+f_6)\delta_6+f_7+(s_7+f_7)\delta_7\right)}_{f_8}(1+\delta_8)+s_8\delta_8+s_8\\
& = f_8+ (s_8+f_8)\delta_8 +s_8,
\end{align*}
with child-error
\begin{align*}
f_8&=f_6+f_7+(s_6+f_6)\delta_6+(s_7+f_7)\delta_7\\
&=\sum_{j=2}^5{(s_j+f_j)\delta_j}+(s_6+f_6)\delta_6+(s_7+f_7)\delta_7
=\sum_{j=2}^7{(s_j+f_j)\delta_j}.
\end{align*}
The total error is
\begin{align*}
e_8&= f_8+ (s_8+f_8)\delta_8 
=\sum_{j=2}^7{(s_j+f_j)\delta_j}+ (s_8+f_8)\delta_8 
=\sum_{j=2}^8{(s_j+f_j)\delta_j}.
\end{align*}
\end{enumerate}
\end{myexample}

\subsection{Setting up martingales on computational trees}\label{s_gprob2}
We derive a probabilistic bound (Lemma~\ref{lemma:fBound})
for the child-errors in Lemma~\ref{lemma:forwardErrorRec},
followed by two types of probabilistic bounds for the error in Algorithm~\ref{alg:sum}: one in terms of a recurrence relation (Theorem~\ref{thm:model2Theorem} and Corollary~\ref{c_28}) and a second in closed form
(Theorem~\ref{thm:model2Analysis} and Corollary~\ref{c_210}).

We introduce our first probability 
parameter $\eta$ which controls
terms  of order two and higher in $e_n$, and
guarantees, with probability
at least $1-\eta$, that all child errors
$|f_k|$ are simultaneously bounded.

\begin{mylemma} \label{lemma:fBound}
Let $L$ be the number of nodes in the computational tree whose children are both leaves, and let $\tilde{n}\equiv n-L-1$ be the number of nodes with at least one non-leaf child. Without loss of generality let nodes $2,\ldots,L+1$ be the ones with two leaf children. Let $0<\eta<1$ and $\lambda_{\tilde{n},\eta} \equiv  \sqrt{2\ln(2\tilde{n}/\eta)}$. Then under (\ref{model:second}), with probability at least $1-\eta$, the child errors
in (\ref{eqn:frecurrence}) are bounded by
\begin{align*}
|f_k|\leq F_{k,\tilde{n},\eta}, \qquad 2\leq k\leq n,
\end{align*}
where
\begin{equation}\label{eqn:FBoundRecurrence}
    F_{k,\tilde{n},\eta}= \begin{cases}
    0 & 2\leq k \leq L+1, \\
    \lambda_{\tilde{n},\eta}u\left(\sum_{j\prec k}(|s_j|+F_{j,\tilde{n},\eta})^2\right)^{1/2} & L+2\leq k \leq n,
    \end{cases}
\end{equation}
\end{mylemma}

\begin{proof}
This is an induction proof over $k$ and the failure probability $\eta$. 

\paragraph{Induction basis $2\leq k \leq L+1$}
Since the inputs (leaves) are assumed to be exact, $f_k\equiv 0$ in (\ref{eqn:frecurrence}). Thus 
$|f_k|\leq F_{k,\tilde{n},\eta}$ clearly holds. 
    
\paragraph{Induction hypothesis}
Assume that the $k-2$  bounds
\begin{align*}
|f_j|\leq F_{j,\tilde{n},\eta}, \qquad 2\leq j\leq k-1\end{align*}
hold simultaneously with probability at least $1-\frac{k-L-2}{\tilde{n}}\eta$. 
\paragraph{Induction step}
Move the precedence relation $j\prec k$ inside the sum, 
in order to write the child-error recurrence \eqref{eqn:frecurrence}
as a contiguous sum,
    \begin{equation*}
        f_k = \sum_{j=2}^{k-1}{(s_j+f_j)\delta_j\mathds{1}_{j\prec k}}.
    \end{equation*}
With $\delta_1= 0$, the sequence  
$Z_1\equiv 0$, $Z_i \equiv \sum_{j=2}^i(s_j + f_j)\delta_j\mathds{1}_{j\prec k}$,
$2\leq i\leq k-1$,
   is a martingale with respect to $\delta_1,\ldots, \delta_{k-1}$. 
According to the induction hypothesis, the $k-2$ bounds  
   \begin{equation*}
       |Z_i - Z_{i-1}| \leq \begin{cases} u(|s_i| + F_{i,\tilde{n},\eta}) & i \prec k,\\
       0 & i \nprec k, 
       \end{cases}\qquad  2\leq i \leq k-1, 
   \end{equation*}
   hold simultaneously
with probability at least $1-\tfrac{k-L-2}{\tilde{n}}\eta$. 
Since $f_k=Z_{k-1}-Z_1$,
setting $\delta=\eta/\tilde{n}$, Lemma~\ref{l_azuma} implies that with probability at least $1-\delta$
\begin{equation*}
    |f_k| \leq \lambda_{\tilde{n},\eta}u\left(\sum_{j\prec k} (|s_j| + F_{j,\tilde{n},\eta})^2\right)^{1/2} = F_{k,\tilde{n},\eta}.
\end{equation*}
 So $|f_j|\leq F_{j,\tilde{n},\eta}$ hold simultaneously for $2\leq j \leq k$ with probability at least $1-\frac{k-L-1}{\tilde{n}}\eta$. By induction, $|f_k|\leq F_{k,\tilde{n},\eta}$ holds for $2\leq k\leq n$ with probability at 
 least $1-\eta$. 
 \end{proof}
 
Sequential summation has $L=1$ nodes
both of whose children are leaves, while pairwise summation has $L=\lfloor n/2\rfloor$. 
 
Finally we are ready for setting up a martingale on a computational tree,
where a second probability parameter $\delta$ controls the first-order terms in $e_n$.

 \begin{theorem}\label{thm:model2Theorem}
Let  $0<\eta < 1$; $0 < \delta < 1-\eta$;
 and $F_{j,\tilde{n},\eta}$ defined as in \eqref{eqn:FBoundRecurrence}.
  Then under the model (\ref{model:second}), with probability at least $1-(\delta+\eta)$, the error in Algorithm \ref{alg:sum} is bounded by
  \begin{equation}\label{eqn:model2errBound}
      |e_n| \leq u\sqrt{2\ln(2/\delta)}\left(\sum_{j=2}^n(|s_j|+F_{j,\tilde{n},\eta})^2\right)^{1/2}.
  \end{equation}
 \end{theorem}
 \begin{proof}
 Write the error as in \eqref{eqn:forwardErrorRec}, 
    \[e_n = \sum_{j=2}^n(s_j+f_j)\delta_j.\]
    With $\delta_1=0$, the sequence $Z_1\equiv 0$, $Z_i\equiv \sum_{j=2}^i(s_j+f_j)\delta_j$, $2\leq i \leq n$, is a martingale with respect to $\delta_1,\ldots,\delta_n$. 
 Lemma~\ref{lemma:fBound} implies that with probability at least~$1-\eta$, the bounds $|f_j|\leq F_{j,\tilde{n},\eta}$ hold simultaneously for $2\leq j \leq n$. Thus with probability at least $1-\eta$, 
 the martingale differences are bounded by
    \[|Z_i - Z_{i-1}| = |(s_i+f_i)\delta_i| \leq u(|s_i|+F_{i,\tilde{n},\eta}),\qquad 
   2\leq i\leq n.\] 
   At last, Lemma \ref{l_azuma1} implies that 
   \eqref{eqn:model2errBound} holds with
 probability at least $1-(\delta+\eta)$.  \end{proof}
 
Below is a simpler bound that holds for every summation algorithm and
does not require knowledge of the number $L$
of nodes with two leaf children. Its first-order version  
illustrates the absence of $\eta$ from the first-order error term. 
 
\begin{mycorollary}\label{c_28}
Let $0<\eta < 1$ and $0 < \delta < 1-\eta$. Then under (\ref{model:second}), with probability at least $1-(\delta+\eta)$, the error in Algorithm \ref{alg:sum} is bounded by
\begin{align*}
      |e_n| &\leq u\sqrt{2\ln(2/\delta)}\left(\sum_{j=2}^n(|s_j|+F_{j,n,\eta})^2\right)^{1/2}\\
    &= u\sqrt{2\ln(2/\delta)}\sqrt{\sum_{k=2}^ns_k^2} + \mathcal{O}(u^2),
    \end{align*}
    where
$F_{2,n,\eta}\equiv 0$ and $F_{k,n,\eta}\equiv\lambda_{n,\eta}u\left(\sum_{j\prec k}(|s_j|+F_{j,n,\eta})^2\right)^{1/2}$, $3\leq k\leq n$.
\end{mycorollary}


\begin{myremark}
We present the following novel approach for proving Theorem~\ref{thm:model2Analysis}.
\begin{enumerate}
\item Write the forward errors $e_k$ in terms of child-errors $f_k$
(see  Lemma~\ref{lemma:forwardErrorRec}).
\item Express each $f_k$ as a martingale in terms of the preceding child-errors, and repeatedly use the Azuma-Hoeffding
inequality in Lemma~\ref{l_azuma}
to bound all of them simultaneously with probability at least $1-\eta$
(see Lemma~\ref{lemma:fBound}).
\item Express the error $e_n$ as a martingale whose bounds depend on the
$f_k$ bounds, and then  derive a bound for 
$|e_n|$ that holds with probability at least $1-(\eta+\delta)$
(see Theorem~\ref{thm:model2Theorem}).
\item Simplify the bound through repeated applications of the triangle inequality.
\end{enumerate}
\end{myremark}

\begin{theorem}\label{thm:model2Analysis}
Let $0<\eta < 1$; $0 < \delta < 1-\eta$, and
$\tilde{n}$ the number of nodes with two non-leaf children.
Then under the model (\ref{model:second}), with probability at least $1-(\delta+\eta)$, the error in Algorithm \ref{alg:sum} is bounded by
\begin{align*}
    |e_n| &\leq u\sqrt{2\ln(2/\delta)}\left(1+\phi_{\tilde{n},h,\eta}\right)\sqrt{\sum_{k=2}^ns_k^2}\\
    &\leq u \sqrt{h}\sqrt{2\ln(2/\delta)}\left(1+\phi_{\tilde{n},h,\eta}\right)\sum_{k=1}^n|x_k|,
\end{align*}
where $h$ is the height of the computational tree and 
\begin{equation}\label{def:phi}
    \phi_{\tilde{n},h,\eta} \equiv \lambda_{\tilde{n},\eta}\sqrt{2h}\,u\,\exp\left(\lambda_{\tilde{n},\eta}^2hu^2\right)\qquad 
\text{with}\qquad    
\lambda_{\tilde{n},\eta} \equiv \sqrt{2\ln(2\tilde{n}/\eta)}.
\end{equation}
\end{theorem}

\begin{proof}
Apply the 2-norm triangle inequality to the sum in Theorem~\ref{thm:model2Theorem},
    \[\left(\sum_{j_1=2}^n(|s_{j_1}|+F_{j_1,\tilde{n},\eta})^2\right)^{1/2} \leq \sqrt{\sum_{k=2}^ns_k^2} + \left(\sum_{j_1\preceq n}F_{j_1,\tilde{n},\eta}^2\right)^{1/2}.\]
    Apply the recurrence for $F_{j,\tilde{n},\eta}$ from  \eqref{eqn:FBoundRecurrence}, followed by the triangle inequality, 
    \begin{align*}
            \left(\sum_{j_1\preceq n}F_{j_1,\tilde{n},\eta}^2\right)^{1/2}
    &=
    \left(\sum_{j_1\preceq n}\sum_{j_2\prec j_1}\lambda_{\tilde{n},\eta}^2u^2(|s_{j_2}|+F_{j_2,\tilde{n},\eta})^2\right)^{1/2} \\
    &\leq \lambda_{\tilde{n},\eta}u \sqrt{\sum_{j_2\prec j_1\preceq n}s_{j_2}^2} + \lambda_{\tilde{n},\eta}u\left(\sum_{j_2\prec j_1\preceq n} F_{j_2,\tilde{n},\eta}^2 \right)^{1/2}\\
    &\leq \lambda_{\tilde{n},\eta}u \sqrt{\binom{h}{1}}\sqrt{\sum_{k=2}^ns_k^2} + \lambda_{\tilde{n},\eta}u\left(\sum_{j_2\prec j_1\preceq n} F_{j_2,\tilde{n},\eta}^2 \right)^{1/2},
    \end{align*}
    where the final inequality follows from the fact that for each 
    index $j_2$, there are at most $h$ possibilities for the index $j_1$, 
    thus each partial sum~$s_k$ appears at most $h$ times. Repeating this and combining the result with Theorem \ref{thm:model2Theorem} 
    shows that with probability at least $1-(\delta + \eta)$ the error
    is bounded by 
    \begin{equation}\label{eqn:model2binomials}
                |e_n| \leq u\sqrt{2\ln(2/\delta)}\left(\sum_{j=0}^h\lambda_{\tilde{n},\eta}^ju^j\sqrt{\binom{h}{j}}\right)\sqrt{\sum_{k=2}^ns_k^2}.
    \end{equation}
    Next, we bound the sum by a simpler expression. Set $\gamma_j \equiv 2^j$ for $1\leq j \leq h$. The Cauchy-Schwarz inequality implies that
    \begin{equation}\label{eqn:cauchySchwarz}
            \left(\sum_{j=1}^h x_j \right)^2  = \left(\sum_{j=1}^h\frac{1}{\sqrt{\gamma_j}}\cdot \sqrt{\gamma_j} x_j\right)^2 \leq  \left(\sum_{j=1}^h \frac{1}{\gamma_j}\right)\left(\sum_{j=1}^h\gamma_j x_j^2\right) \leq \sum_{j=1}^h \gamma_j x_j^2. 
    \end{equation}
Thus, 
\begin{align*}
    \sum_{j=1}^h\lambda_{\tilde{n},\eta}^ju^j\sqrt{\binom{h}{j}} 
    &\leq \left(\sum_{j=1}^h 2^j\lambda_{\tilde{n},\eta}^{2j}u^{2j}\binom{h}{j} \right)^{1/2}
    =  \sqrt{(1+2\lambda_{\tilde{n},\eta}^2u^2)^h - 1}\\
    &\leq \sqrt{\exp\left(2\lambda_{\tilde{n},\eta}^2hu^2\right)-1}\\
    &\leq \sqrt{2\lambda_{\tilde{n},\eta}^2hu^2\exp\left(2\lambda_{\tilde{n},\eta}^2hu^2\right)}
    = \phi_{\tilde{n},h,\eta}. 
\end{align*}
Substituting this bound into \eqref{eqn:model2binomials} gives the desired result. 
\end{proof}

Theorem~\ref{thm:model2Analysis}  implies that with high probability 
the summation error to first order is proportional to $\sqrt{h}$, where 
$h$ is the height of the computational tree. This confirms that even
under the probabilistic model, summation algorithms based on 
shallow computational trees are likely to be more accurate.

\begin{myremark}
The quantity $\phi_{\tilde{n},h,\eta}$ appears only in second and higher order terms 
of the error, and might possibly become significant only if the computational tree is deep enough so that $\lambda_{\tilde{n},\eta}\sqrt{2h}u\approx 1$.
However, the effect of $\eta$ on the overall bound, even under adverse
circumstances, is negligible. 

Consider single precision computation where $u=2^{-24}\approx 5.96\cdot 10^{-8}$. Assume an extreme problem size $n=10^{10}$ with a computational tree of maximal height $h=n$, and a tremendously strict probability $\eta=10^{-32}$. Then 
    $\lambda_{\tilde{n},\eta}\approx 14.0$, and $\exp\left(\lambda_{\tilde{n},\eta}^2hu^2\right)=1$ to three digits,
    so that  the total contribution of the higher order terms is 
    merely a factor of $1+\phi_{\tilde{n},h,\eta}<1.12$.
   \end{myremark}
    
In the special case of sequential summation, 
the first bound in Theorem~\ref{thm:model2Analysis} is stronger than \cite[Theorem 2.4]{higham2020sharper}, while the second bound shows agreement to first order.

For completeness, we present
a simpler bound that holds for all summation algorithms
and does not require knowledge of the number of nodes $L$ with two 
leaf children.

\begin{mycorollary}\label{c_210}
Let $0<\eta < 1$; $0 < \delta < 1-\eta$.
Then under the model \eqref{model:second}, with probability at least $1-(\delta+\eta)$, the error in Algorithm \ref{alg:sum} is bounded by
\begin{align*}
    |e_n| &\leq u\sqrt{2\ln(2/\delta)}\left(1+\phi_{n,h,\eta}\right)\sqrt{\sum_{k=2}^ns_k^2}\\
    &\leq u \sqrt{h}\sqrt{2\ln(2/\delta)}\left(1+\phi_{n,h,\eta}\right)\sum_{k=1}^n|x_k|,
\end{align*}
where $h$ is the height of the computational tree and 
\begin{equation}\label{def:phi1}
    \phi_{n,h,\eta} \equiv \lambda_{n,\eta}\sqrt{2h}\,u\,\exp\left(\lambda_{n,\eta}^2hu^2\right)\qquad 
\text{with}\qquad    
\lambda_{n,\eta} \equiv \sqrt{2\ln(2n/\eta)}.
\end{equation}
\end{mycorollary}

\section{Shifted summation}\label{s_center}
We present an algorithm for shifted
summation 
(Algorithm~\ref{alg:shiftSum}) which
centers the inputs $x_j$, and derive 
a probabilistic error bound (Theorem~\ref{thm:shifted}). 

Shifted summation is motivated by 
work  in computer architecture \cite{CDRS21,DSC19} and formal methods
for program verification \cite{Lohar19}
where not only the roundoffs  but also the inputs are interpreted as 
random variables sampled from 
 some distribution. Then one can compute statistics for
 the total roundoff error and estimate the probability that it is bounded by $tu$
 for a given $t$.
 
Probabilistic bounds for random inputs are derived in 
\cite{higham2020sharper}, with improvements in \cite{hallman2021refined},
to show that sequential summation is accurate for inputs 
$x_j$ that are tightly clustered around zero.
As a consequence, accuracy can be improved 
by shifting the inputs to have zero mean, which is affordable in the context 
of matrix multiplication \cite[Section 4]{higham2020sharper}.

\begin{algorithm}\caption{Shifted General Summation}\label{alg:shiftSum}
\begin{algorithmic}[1]
\REQUIRE Floating point numbers $x_1,\ldots,x_n$; shift $c$
\ENSURE $s_n = \sum_{k=1}^n{x_k}$
\FOR{$k=1:n$}
\STATE $y_k=x_k-c$
\ENDFOR
\STATE $y_{n+1}=nc$
\STATE $t_n$ = output of Algorithm \ref{alg:sum} applied
to $y_1,\ldots,y_n$
\RETURN $s_n = t_n +y_{n+1}$
\end{algorithmic}
\end{algorithm}

Our Algorithm~\ref{alg:shiftSum} extends
the shifted algorithm for sequential summation 
\cite[Algorithm 4.1]{higham2020sharper} to general summation. Figure \ref{fig:Centering} shows the computational tree for $n=2$.

The pseudo-code in
Algorithm~\ref{alg:shiftSum} is geared 
towards exposition. In practice, one shifts the $x_k$ immediately prior to the
summation, to avoid allocating additional storage for $y_k=x_k-c$.
The ideal choice for centering is the empirical mean $c=s_n/n$.
A simpler approximation is $c=(\min_k{x_k}+\max_k{x}_k)/2$.

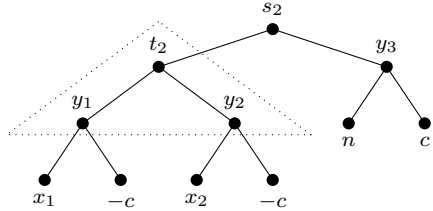
\begin{figure}
   \centering
     \begin{tikzpicture}[scale=1.0]
	\tikzstyle{every node}+=[inner sep=0pt]
	\fill (0,0) circle (0.08) node [below=5] {$x_1$}; 
	\fill (1,0) circle (0.08) node [below=5] {$-c$};  
	\fill (2,0) circle (0.08) node [below=5] {$x_2$}; 
	\fill (3,0) circle (0.08) node [below=5] {$-c$}; 
	\fill (4,0.75) circle (0.08) node [below=5] {$n$}; 
	\fill (5,0.75) circle (0.08) node [below=5] {$c$}; 
	\draw (0,0)-- (0.5,0.75) -- (1,0); 
	\draw (2,0) -- (2.5,0.75) -- (3,0); 
	\draw (0.5,0.75) -- (1.5,1.5) -- (2.5,0.75); 
	\draw (4,0.75) -- (4.5,1.5) -- (5,0.75);
	\draw (1.5,1.5) -- (3,2) -- (4.5,1.5); 
	\fill (0.5,0.75) circle (0.08) node [above=5] {$y_1$};
	\fill (2.5,0.75) circle (0.08) node [above=5] {$y_2$};
	\fill (1.5,1.5) circle (0.08) node [above=5] {$t_2$};
	\fill (4.5,1.5) circle (0.08) node [above=5] {$y_3$};  
	\fill (3,2) circle (0.08) node [above=5] {$s_2$};
	\draw[dotted] (1.5,2.1) -- (-0.5,0.6) -- (3.5,0.6) -- (1.5,2.1); 
	\end{tikzpicture}

    \caption{Computational tree for shifted summation, $n=2$. The dotted boundary delineates the inputs of and summations computed by the call to Algorithm \ref{alg:sum} in line~5  of Algorithm~\ref{alg:shiftSum}.}
    \label{fig:Centering}
\end{figure}
    
Error bounds for Algorithm \ref{alg:shiftSum} follow almost directly from the ones for Algorithm~\ref{alg:sum}. Figure \ref{fig:Centering} shows the associated computational tree for $n=2$. It has $4n+3$ vertices, and its height is equal to two plus the height of the tree in the call to Algorithm~\ref{alg:sum}. The one twist is the additional multiplication $y= nc$, but if $n$ and $c$ can be stored exactly then the error analysis remains the same.\footnote{ If $n$ does not admit an exact floating point 
representation, then we could append an additional node for the 
artificial `addition' $n+0$, which simulates the rounding of $n$.}

\begin{theorem}\label{thm:shifted}
Let $0<\eta < 1$; $0<\delta<1-\eta$. Then under the model \eqref{model:second}, with probability at least $1-(\delta+\eta)$, the error in Algorithm~\ref{alg:shiftSum} is bounded by 
\begin{align*}
    |e_n| &\leq u\sqrt{2\ln(2/\delta)}\left(1+\phi_{n,h,\eta}\right)\sqrt{s_n^2 + \sum_{k=2}^nt_k^2 + \sum_{k=1}^{n+1}y_k^2}\\
    &\leq u\sqrt{2\ln(2/\delta)}\left(1+\phi_{n,h,\eta}\right)\left(n|c| +\sqrt{h}\sum_{k=1}^n{(|x_k-c|+|x_k|)}\right),
\end{align*}
where $h$ is the height of the computational tree and $\phi_{n,h,\eta}$ is defined in \eqref{def:phi1}.
\end{theorem}

With regard to the factor $\lambda_{n,\eta}\equiv \sqrt{2\ln(2n/\eta)}$ in $\phi_{n,h,\eta}$, the tree for Algorithm \ref{alg:shiftSum} has $L=n$ nodes both of whose children are leaves.

\section{Compensated sequential summation}\label{s_compensated}

Our approach is not restricted to algorithms whose computational graphs are trees, and we demonstrate its versatility by analyzing the forward error for compensated sequential summation (Algorithm~\ref{alg:compensated}). After deriving exact error expressions (Section~\ref{s_cexact}) and
bounds that hold to second order (Section~\ref{s_cfirst}), we derive an exact probabilistic bound  (Section~\ref{s_compProbBnd}). 

Algorithm~\ref{alg:compensated} is the formulation \cite[Theorem 8]{goldberg1991every}
of the `Kahan Summation Formula' \cite{kahan1965pracniques}. A version with opposite signs 
is presented in \cite[Algorithm 4.2]{higham2002accuracy}.

\begin{algorithm}\caption{Compensated Summation \cite[Theorem 8]{goldberg1991every} \cite[page 9-4]{Kahan73}}\label{alg:compensated}
\begin{algorithmic}[1]
\REQUIRE Floating point numbers $x_1, \ldots, x_n$
\ENSURE $s_n = \sum_{k=1}^n{x_k}$
\STATE $s_1=x_1$, $c_1=0$
\FOR{$k=2:n$}
\STATE $y_k=x_k-c_{k-1}$
\STATE $s_k=s_{k-1}+y_k$
\STATE $c_k=(s_k-s_{k-1}) -y_k$
\ENDFOR
\RETURN $s_n$
\end{algorithmic}
\end{algorithm}

Following \cite[page 9-5]{Kahan73} and additionally defining the computed terms $\widehat{z}_k$, our finite precision model of 
Algorithm~\ref{alg:compensated} is 
\begin{align}\label{e_model1}
\begin{split}
    \widehat{s}_1 &= s_1=x_1, \qquad \widehat{c}_1= 0, \qquad \eta_2=0 \\
\widehat{y}_k&= (x_k-\widehat{c}_{k-1})(1+\eta_k), \qquad 2\leq k\leq n\\
\widehat{s}_k &= (\widehat{s}_{k-1} +\widehat{y}_k)(1+\sigma_k)\\
\widehat{z}_k &= (\widehat{s}_k-\widehat{s}_{k-1})(1+\delta_k)\\
 \widehat{c}_k&= \left(\widehat{z}_k - 
 \widehat{y}_k\right)(1+\beta_k),
\end{split}
\end{align}

\subsection{Error expressions}\label{s_cexact}
Mimicking the strategy for  general summation, we derive an analogue of Lemma \ref{lemma:forwardErrorRec} for compensated summation. 
We use single dots to represent individual forward errors\footnote{The dots do not refer to differentiation!},
\begin{equation}\label{def:compFwdErr}
     \dot{y}_k \equiv \widehat{y}_k - x_k,\qquad \dot{s}_k \equiv \widehat{s}_k - s_k,\qquad \dot{z}_k \equiv \widehat{z}_k - x_k, \qquad \dot{c}_k \equiv \widehat{c}_k, 
\end{equation}
and double dots to represent child-errors,
\begin{equation}\label{def:compChiErr}
    \ddot{y}_k \equiv -\dot{c}_{k-1}, \qquad \ddot{s}_k \equiv \dot{s}_{k-1}+\dot{y}_k, \qquad \ddot{z}_k \equiv \dot{s}_k -\dot{s}_{k-1}, \qquad \ddot{c}_k \equiv \dot{z}_k - \dot{y}_k. 
\end{equation}
The relations \eqref{e_model1} imply the forward error recursions 
\begin{subequations}
\begin{align}
    \dot{y}_k &= (x_k +\ddot{y}_k)\eta_k + \ddot{y}_k, \label{eqn:comp_ydot}\\
    \dot{s}_k &= (s_k + \ddot{s}_k)\sigma_k + \ddot{s}_k, \label{eqn:comp_sdot}\\
    \dot{z}_k &= (x_k +\ddot{z}_k)\delta_k + \ddot{z}_k, \label{eqn:comp_zdot}\\
    \dot{c}_k &= \ddot{c}_k\beta_k + \ddot{c}_k. \label{eqn:comp_cdot}
\end{align}
\end{subequations}
Now we derive recurrence relations for the child-errors. 
Fortunately, the recurrences for $\ddot{y}_k$, $\ddot{z}_k$, and 
$\ddot{c}_k$ are mercifully short, with a length independent
of $k$.

\begin{theorem}\label{thm:comp_recurrences}
The child-errors in Algorithm \ref{alg:compensated} equal 
\begin{equation}\label{eqn:cbase}
    \ddot{y}_2 = 0, \qquad \ddot{s}_2 = 0, \qquad \ddot{z}_2 = s_2\sigma_2, \qquad \ddot{c}_2 = (x_2+\ddot{z}_2)\delta_2 + s_2\sigma_2, 
\end{equation}
and for $3\leq k \leq n$, 
\begin{subequations}
\begin{align}
    \ddot{y}_k &= -\ddot{c}_{k-1}(1+\beta_{k-1}), \label{eqn:yrec}\\
    \ddot{s}_k &= \sum_{j=3}^k (x_j+\ddot{y}_j)\eta_j - \ddot{c}_{j-1}\beta_{j-1} - (x_{j-1}+\ddot{z}_{j-1})\delta_{j-1}, \label{eqn:srec}\\
    \ddot{z}_k &= (s_k + \ddot{s}_k)\sigma_k + (x_k+\ddot{y}_k)\eta_k + \ddot{y}_k,\label{eqn:zrec}\\
    \ddot{c}_k &= (x_k + \ddot{z}_k)\delta_k + (s_k + \ddot{s}_k)\sigma_k. \label{eqn:crec}
\end{align}
\end{subequations}
\end{theorem}
\begin{proof}
    First, \eqref{eqn:yrec} follows directly from \eqref{def:compChiErr} and \eqref{eqn:comp_cdot}. Second, 
    \begin{align*}
        \ddot{c}_k &= \dot{z}_k-\dot{y}_k &&\text{by \eqref{def:compChiErr}}\\
                   &= (x_k+\ddot{z}_k)\delta_k + \ddot{z}_k - \dot{y}_k&&  \text{by \eqref{eqn:comp_zdot}}\\
                   &= (x_k + \ddot{z}_k)\delta_k + \dot{s}_k-\dot{s}_{k-1}-\dot{y}_k&& \text{by \eqref{def:compChiErr}}\\
                   &= (x_k + \ddot{z}_k)\delta_k + (s_k+\ddot{s}_k)\sigma_k + \ddot{s}_k-(\dot{s}_{k-1}+\dot{y}_k)&&\text{by \eqref{eqn:comp_sdot}}\\
                   &= (x_k + \ddot{z}_k)\delta_k + (s_k+\ddot{s}_k)\sigma_k,&&\text{by \eqref{def:compChiErr}}\\
        \intertext{which establishes \eqref{eqn:crec}. Third,}
        \ddot{s}_k &= \dot{s}_{k-1} + \dot{y}_k &&\text{by \eqref{def:compChiErr}}\\
                  &= \ddot{s}_{k-1} + (s_{k-1}+\ddot{s}_{k-1})\sigma_{k-1} + (x_k+\ddot{y}_k)\eta_k + \ddot{y}_k &&\text{by \eqref{eqn:comp_ydot}, \eqref{eqn:comp_sdot}}\\
                  &= \ddot{s}_{k-1} + (s_{k-1}+\ddot{s}_{k-1})\sigma_{k-1} + (x_k+\ddot{y}_k)\eta_k - \ddot{c}_{k-1}(1+\beta_{k-1}) &&\text{by \eqref{eqn:yrec}} \\
                  &= \ddot{s}_{k-1} + (x_k+\ddot{y}_k)\eta_k - \ddot{c}_{k-1}\beta_{k-1} - (x_{k-1}+\ddot{z}_{k-1})\delta_{k-1},&&\text{by \eqref{eqn:crec}}
    \intertext{and unraveling the recurrence yields \eqref{eqn:srec}. Finally, }
        \ddot{z}_k &= \dot{s}_k-\dot{s}_{k-1} &&\text{by \eqref{def:compChiErr}}\\
                &= (s_k+\ddot{s}_k)\sigma_k + \ddot{s}_k - \dot{s}_{k-1} &&\text{by \eqref{eqn:comp_sdot}}\\
                &= (s_k+\ddot{s}_k)\sigma_k + \dot{y}_k &&\text{by \eqref{def:compChiErr}}\\
                &= (s_k+\ddot{s}_k)\sigma_k + (x_k+\ddot{y}_k)\eta_k + \ddot{y}_k.&&\text{by \eqref{eqn:comp_ydot}}
    \end{align*}
Recalling that $\eta_2=0$, it is straightforward to check \eqref{eqn:cbase} separately. 
\end{proof}

\subsection{Second order deterministic bound}\label{s_cfirst}
We present a second-order expression (Corollary~\ref{c_cs5}) 
for the error in 
Algorithm~\ref{alg:compensated}, and discuss the discrepancy with 
several existing bounds (Remark~\ref{r_cfirst}).

The expressions below suggest that the errors in the `correction' steps
3 and~5 of Algorithm~\ref{alg:compensated} dominate the first order 
terms of the summation error.

\begin{mycorollary}\label{c_cs5}
With assumptions (\ref{e_model1}), let
$\mu_k\equiv \eta_k-\delta_k$, $2\leq k\leq n-1$,
and $\mu_n\equiv \eta_n$. Then the error in
Algorithm~\ref{alg:compensated} up to second order equals
\begin{align*}
e_n=\widehat{s}_n-s_n=  \dot{s}_n = s_n\sigma_n + (1+\sigma_n)\sum_{k=2}^n{x_k\mu_k}
         &-\sum_{k=2}^{n-1}{s_k\sigma_k(\mu_{k+1}+\beta_k+\delta_k)}\\
        &- \sum_{k=2}^{n-1}{x_k\delta_k(\mu_{k+1}+\beta_k+\eta_k)}
        + \mathcal{O}(u^3),
\end{align*}
and the computed sum equals
\begin{equation}\label{eqn:c_phibound}
    \widehat{s}_n = \sum_{k=1}^n(1+\rho_k)x_k, \qquad |\rho_k|\leq 3u + [4(n-k)+6]u^2+\mathcal{O}(u^3).
\end{equation}
 \end{mycorollary}

\begin{proof}
Truncate the expressions for $\ddot{y}_k$, $\ddot{z}_k$, and $\ddot{c}_k$ to first order, and substitute them into \eqref{eqn:srec}.  
\end{proof}

\begin{myremark}\label{r_cfirst}
The error bounds for compensated summation have sometimes been misstated in the literature. In contrast to \eqref{eqn:c_phibound}, the expressions in \cite[Theorem 8]{goldberg1991every},  \cite[(4.8)]{higham2002accuracy} and
\cite[Exercise 19 in Section 4.2.2]{Knuth2}
are equal to
\begin{equation*}
\widehat{s}_n=\sum_{k=1}^n{(1+\rho_k)x_k}\qquad \text{where}\quad
|\rho_k|\leq 2u + \mathcal{O}(nu^2).
\end{equation*}
It appears that this expression does not properly account for the final error $\sigma_n$. In comparison, \cite[page 9-5]{Kahan73} correctly states that
\begin{equation*}
    \widehat{s}_n+\widehat{c}_n=\sum_{k=1}^n{(1+\rho_k)x_k}\qquad \text{where}\quad
|\rho_k|\leq 2u + \mathcal{O}((n-k)u^2).
\end{equation*}
\end{myremark}

\subsection{Probabilistic bounds}\label{s_compProbBnd}
We derive probabilistic bounds for the child-errors in compensated summation (Lemma~\ref{lemma:sddot})
and derive
a bound on the summation error
in terms of the child-error bounds
(Theorem~\ref{lemma:compProbErr}),
which is, however, difficult to interpret.
Thus, we express the child-error bounds mostly
in terms of the partial sums
(Lemma~\ref{lemma:alphaBound}), which leads
to an alternative probabilistic
bound (Theorem~\ref{thm:comp_prob_err}).

We start with a probabilistic analogue of Lemma~\ref{lemma:fBound}. The
generic strategy would be to write each
child-error
in terms of a martingale involving the previous child-errors, and to bound them probabilistically with the Azuma-Hoeffding inequality (Lemma~\ref{lemma:Azuma}). 
Instead, we found it easier here 
to bound $\ddot{s}_k$ 
Lemma~\ref{lemma:Azuma}, and then apply the triangle inequality to $\ddot{y}_k$, $\ddot{z}_k$, and $\ddot{c}_k$.

\begin{mylemma}\label{lemma:sddot}
Let $\sigma_2,\delta_2,\beta_2,\eta_3,\sigma_3,\delta_3,\ldots,\eta_n$ in 
(\ref{e_model1}), (\ref{def:compFwdErr})
and (\ref{def:compChiErr})
be mean-independent zero-mean random variables, $0<\eta<1$, 
  and $\lambda_{n,\eta} \equiv \sqrt{2\ln(2n/\eta)}$. With probability at least $1-\eta$, the following bounds hold simultaneously:
  \begin{equation}\label{eqn:comp_errorBounds}
           |\ddot{y}_k|\leq Y_k, \qquad |\ddot{s}_k| \leq S_k, \qquad |\ddot{z}_k| \leq Z_k, \qquad |\ddot{c}_k| \leq C_k, \qquad 2\leq k \leq n, 
  \end{equation}
 where the quantities $Y_k$, $S_k$, $Z_k$, $C_k$ are defined by 
  \begin{equation}\label{eqn:bbase}
          Y_2\equiv 0,\qquad S_2 \equiv 0,\qquad Z_2 \equiv u|s_2|, \qquad C_2 \equiv u(|x_2|+Z_2) + u|s_2|,
  \end{equation}
  and\footnote{The bounds depend on $n$ and $\eta$, but 
we omit the subscripts, and simply write $S_k$ instead 
of $S_{k,n,\eta}$.} for $3\leq k \leq n$, 
  \begin{subequations}
    \begin{align}
     Y_k &\equiv C_{k-1}(1+u), \label{eqn:byrec}\\
    S_k &\equiv \lambda_{n,\eta}u\Biggl(\sum_{j=3}^k{\left( (|x_j|+Y_j)^2 + C_{j-1}^2 + (|x_{j-1}|+Z_{j-1})^2\right)}\Biggr)^{1/2},\label{eqn:bsrec}\\
  Z_k & \equiv u(|s_k|+S_k) + u(|x_k| + Y_k) + Y_k, \label{eqn:bzrec}\\
  C_k &\equiv  u(|x_k|+Z_k) +  u(|s_k|+S_k). \label{eqn:bcrec}
  \end{align}
  \end{subequations}
\end{mylemma}

\begin{proof}
This is an induction proof over $k$ and the failure probability $\eta$. 
\paragraph{Induction basis $k=2$} From \eqref{eqn:cbase} in Theorem \ref{thm:comp_recurrences} follows that (\ref{eqn:bbase})
holds deterministically.

\paragraph{Induction hypothesis} Assume that the bounds \eqref{eqn:comp_errorBounds} hold simultaneously for $2\leq j \leq k-1$ 
with probability at least $1-(k-1)\eta/n$. 
\paragraph{Induction step} The induction hypothesis
implies that $|\ddot{c}_{k-1}|\leq C_{k-1}$ holds with probability at least $1-(k-1)\eta$. From \eqref{eqn:yrec}, it follows that 
\[|\ddot{y}_k| = |\ddot{c}_{k-1}(1+\beta_{k-1})|\leq C_{k-1}(1+u)=Y_k\]
always holds. 

The expression \eqref{eqn:srec} for $\ddot{s}_k$ can be written as a martingale with respect to $\sigma_2,\delta_2,\beta_2,\eta_3,\sigma_3,\delta_3,\ldots,\eta_k$. By the induction hypothesis, the bounds 
\begin{align*}
|(x_j+\ddot{y}_j)\eta_j| &\leq u(|x_j|+Y_j), \qquad 3\leq j\leq k\\
|\ddot{c}_{j-1}\beta_{j-1}| &\leq u\,C_{j-1}, \\
    |(x_{j-1}+\ddot{z}_{j-1})\delta_{j-1}| &\leq u(|x_{j-1}|+Z_{j-1})
\end{align*}
all hold simultaneously with probability at least  $1-(k-1)\eta/n$. 
Lemma~\ref{lemma:Azuma} then implies that $|\ddot{s}_k|\leq S_k$ 
holds with probability at least $1-\eta/n$.

The bounds $|\ddot{z}_k| \leq Z_k$ and $|\ddot{c}_k|\leq C_k$ 
always hold, due to \eqref{eqn:zrec} and \eqref{eqn:crec}. 
\end{proof}

The following probabilistic bound expressed the error in compensated summation in terms of the bounds for child-errors.

\begin{theorem}\label{lemma:compProbErr}
Let $\sigma_2,\delta_2,\beta_2,\eta_3,\ldots,\eta_n,\sigma_n$ in
(\ref{e_model1}), (\ref{def:compFwdErr})
and (\ref{def:compChiErr})
be mean-independent zero-mean random variables, $0<\eta<1$, $0 < \delta < 1-\eta$, and $\lambda_{n,\eta}\equiv \sqrt{2\ln(2n/\eta)}$.
Then under (\ref{model:second}), with probability at least $1-(\delta+\eta)$, the error in Algorithm \ref{alg:compensated} is bounded by 
\begin{align*}
      \begin{split}
            |e_n| &\leq u\sqrt{2\ln(2/\delta)}\Biggl((s_n+S_n)^2 + \sum_{j=3}^n\left( (|x_j|+Y_j)^2+C_{j-1}^2 + (|x_{j-1}|+Z_{j-1})^2\right) \Biggr)^{1/2}, 
  \end{split}
\end{align*}
where $Y_j$, $S_j$, $Z_j$, $C_j$, $2\leq j \leq n$, are defined in Lemma \ref{lemma:sddot}. 
\end{theorem}

\begin{proof}
Keeping in mind that that $e_n=\dot{s}_n$,
substitute \eqref{eqn:srec} into \eqref{eqn:comp_sdot},
 bound the magnitude of the summands with probability at least $1-\eta$
 via \ref{lemma:sddot}
 and apply Lemma~\ref{l_azuma1} with additional probability $\delta$.
 This derivation mirrors the proof of Theorem \ref{thm:model2Theorem} which
 relies on Lemma~\ref{lemma:fBound} to
 bound the magnitude of the summands in the martingale.
 \end{proof}

The significant number of interacting
terms make Theorem \ref{lemma:compProbErr} difficult to interpret, in comparison to 
Theorem \ref{thm:model2Theorem}. The simplest approach at this point would be to truncate the terms $S_k,Y_k,Z_k,C_k$ so that the overall bound holds to second order (or higher, if desired). With Lemma \ref{lemma:alphaBound} and Theorem \ref{thm:comp_prob_err}, we instead show that it is possible to obtain a bound that holds to all orders, at the cost of a more complicated proof.
Consequently we derive an alternative bound in the same manner as before, alternating between the triangle inequality and the following bound.

\begin{mylemma}\label{lemma:alphaBound}
 There is a constant $\alpha = \sqrt{6} + \mathcal{O}(u)$, so that the terms 
in Lemma~\ref{lemma:sddot} can be bounded by
 \[\left(\sum_{j=3}^k\left(Y_j^2 + C_{j-1}^2 + Z_{j-1}^2\right)\right)^{1/2}\leq \alpha u \left(\sum_{j=2}^{k-1}(|s_j|+|x_j|+S_j)^2\right)^{1/2}, \qquad 3\leq k\leq n.\]
\end{mylemma}

\begin{proof}
 The precise value of $\alpha$ is derived in the Appendix~\ref{apx:alpha}.
 \end{proof}

The next bounds for compensated summation is expressed in terms of partial sums and inputs. 

\begin{theorem}\label{thm:comp_prob_err}
Let $\sigma_2, \delta_2, \beta_2, \eta_3, \sigma_3, \delta_3, \ldots, \eta_n, \sigma_n$ be mean-independent zero-mean random variables, $0<\eta<1$, $0 < \delta < 1-\eta$, and
$\lambda_{n,\eta}\equiv \sqrt{2\ln(2n/\eta)}$. Then with probability at least $1-(\delta+\eta)$, the error in Algorithm \ref{alg:compensated} is bounded by

\begin{align*}
     |e_n|&\leq u\sqrt{2\ln(2/\delta)}
\left(|s_n| + \gamma(\sqrt{2}+\alpha u)\sqrt{\sum_{k=2}^nx_k^2}+\gamma\alpha u\sqrt{\sum_{k=2}^ns_k^2}\right) \\
&\leq u\sqrt{2\ln(2/\delta)}
\left(1+\sqrt{2}+\sqrt{6}(\sqrt{n}+1) u\right)\sum_{k=1}^n|x_k| + \mathcal{O}(u^3), 
\end{align*}
where 
\begin{align*}
    \alpha &\equiv \frac{\sqrt{1+3(1+u)^2+2(1+u)^4}}{1-u(1+u)^2}= \sqrt{6}+\mathcal{O}(u), \\
    \gamma &\equiv \sqrt{1+\lambda_{n,\eta}^2u^2}\left(1 + \lambda_{n,\eta}\alpha\sqrt{2n}u^2\exp\left(\lambda_{n,\eta}^2\alpha^2nu^4\right) \right) = 1 + \mathcal{O}(u^2). 
\end{align*}
\end{theorem}

\begin{proof}
Remember that $e_n=\dot{s}_n$, and abbreviate the summands in Theorem~\ref{lemma:compProbErr} and in (\ref{eqn:bsrec}) by
 \[
  R_j \equiv \left((|x_j|+Y_j)^2 + C_{j-1}^2 + (|x_{j-1}|+Z_{j-1})^2\right)^{1/2}, \qquad 3\leq j \leq n. 
 \]
We treat $\sum_{j=3}^n{R_j^2}$ as a two-norm and apply
the following inequality for non-negative vectors $c$, $x$, $y$, 
and $z$,
 \begin{align*}
 \left(\|x+y\|_2^2+\|x+z\|_2^2+\|c\|_2^2\right)^{1/2}
 \leq \sqrt{2}\|x\|_2+\|c+y+z\|_2.
 \end{align*}
followed by Lemma \ref{lemma:alphaBound}, two
triangle inequalities, and the definition of $S_j$,
 \begin{align*}
    \left(\sum_{j=3}^n R_j^2 \right)^{1/2}
    &\leq  \sqrt{2}\left(\sum_{k=2}^nx_k^2\right)^{1/2} + \left(\sum_{j=3}^n \left(Y_j^2 + C_{j-1}^2 + Z_{j-1}^2 \right)\right)^{1/2}\\
    &\leq \sqrt{2}\left(\sum_{k=2}^nx_k^2\right)^{1/2} + \alpha u \left(\sum_{j=2}^{n-1} (|s_{j}|+|x_j|+S_j)^2 \right)^{1/2}\\
    &\leq  \sqrt{2}\left(\sum_{k=2}^nx_k^2\right)^{1/2} + \alpha u \left(\sum_{k=2}^{n}(|s_k|+|x_k|)^2\right)^{1/2} + \alpha u  \left(\sum_{j=3}^{n-1} S_j^2 \right)^{1/2}\\
    &\leq  (\sqrt{2}+\alpha u)\left(\sum_{k=2}^nx_k^2\right)^{1/2} + \alpha u \sqrt{\sum_{k=2}^{n}s_k^2} + \lambda\alpha u^2  \left(\sum_{j<j_1\leq n}R_j^2 \right)^{1/2}.
 \end{align*}
Proceed as in the proof of Theorem~\ref{thm:model2Analysis},
 \begin{equation}\label{eqn:c_Rbound}
     \left(\sum_{j=3}^n R_j^2 \right)^{1/2}
    \leq \left(\sum_{j=0}^n(\lambda_{n,\eta}\alpha u^2)^j\sqrt{\binom{n}{j}}\right)\left((\sqrt{2}+\alpha u)\sqrt{\sum_{k=2}^nx_k^2}+\alpha u\sqrt{\sum_{k=2}^ns_k^2}\right),
 \end{equation}
where
 \begin{equation}\label{eqn:c_binBound}
     \sum_{j=1}^n(\lambda_{n,\eta}\alpha u^2)^j\sqrt{\binom{n}{j}}
     \leq 
     \lambda_{n,\eta}\alpha\sqrt{2n}u^2\exp\left(\lambda_{n,\eta}^2\alpha^2nu^4\right). 
 \end{equation}
From Theorem~\ref{lemma:compProbErr}; the inequality 
$(a+b)^2+c^2\leq (a+\sqrt{b^2+c^2})^2$ for $a,b,c\geq 0$; and the 
definition of $S_n$ in (\ref{eqn:bsrec}) follows
\begin{align*}
    |\dot{s}_n|&\leq u\sqrt{2\ln(2/\delta)}\left((s_n+S_n)^2 + \sum_{j=3}^nR_j^2\right)^{1/2}\\
    &\leq u\sqrt{2\ln(2/\delta)}|s_n| + 
    u\sqrt{2\ln(2/\delta)}\left(S_n^2 + \sum_{j=3}^nR_j^2\right)^{1/2}\\
    &= u\sqrt{2\ln(2/\delta)}|s_n| + 
    u\sqrt{2\ln(2/\delta)}\sqrt{1+\lambda_{n,\eta}^2u^2}\left(\sum_{j=3}^nR_j^2\right)^{1/2}.
\end{align*}
Combine this with \eqref{eqn:c_Rbound} and \eqref{eqn:c_binBound}.
\end{proof}

Note that $\gamma$ remains close to 1 as long as $\lambda_{n,\eta}u\ll 1$ and  $\lambda_{n,\eta}\alpha \sqrt{2n}u^2\ll 1$.


\section{Mixed precision}\label{s_mixedPrec}
Mixed-precision algorithms aim to do as much of the computation as possible in a lower precision without significantly degrading the accuracy of the computed result; see the survey \cite{abdelfattah2021survey}.
We extend Corollaries \ref{c_28} and~\ref{c_210} to any number of
precisions (Theorems \ref{t_51} and~\ref{t_52}), present the first
probabilistic error bounds for the mixed precision \texttt{FABsum}
algorithm (Corollary~\ref{c_fabsum}), and end with a heuristic
for designing mixed-precision algorithms (Remark~\ref{r_51}).

The \texttt{FABsum} summation algorithm \cite[Algorithm 3.1]{blanchard2020class} computes the sum $s_n=x_1+\cdots +x_n$ in two stages. First, it splits the inputs into blocks of $b$ numbers,
and sums each block with a fast summation algorithm, say in low precision. Second, it sums the results with an accurate summation algorithm, say in high precision 
or with compensated summation. 
We extend our approach to mixed precision, and derive the first rigorous probabilistic error bounds for \texttt{FABsum}. 
Our computational model is very general, so that,
in theory, each operation can be evaluated in a different precision.

\paragraph{Probabilistic model for sequences of roundoffs in mixed precision} 
Extend model~(\ref{model:second})
for roundoffs in terms of mean-independent zero-mean random variables
$\delta_k$ by assuming in addition that each $\delta_k$
can be a roundoff in a different precision
$u_k$, that is,
$|\delta_k|\leq u_k$, $1\leq k\leq n$.

Below are the straightforward generalizations of Corollaries \ref{c_28}
and~\ref{c_210}.

\begin{theorem}\label{t_51}
   Let $0<\eta < 1$, $0 < \delta < 1-\eta$, and
$\lambda_{n,\eta} \equiv \sqrt{2\ln(2n/\eta)}$.
   Then under (\ref{model:second}), with probability at least $1-(\delta+\eta)$, the error in Algorithm \ref{alg:sum} is bounded by
  \begin{equation}
      |e_n| \leq \sqrt{2\ln(2/\delta)}\left(\sum_{j=2}^nu_j^2(|s_j|+F_{j,n,\eta})^2\right)^{1/2},
  \end{equation}
  where $F_{j,n,\eta}$ are defined by the recurrence 
  \begin{equation}
      F_{2,n,\eta}\equiv 0, \qquad F_{k,n,\eta}\equiv \lambda_{n,\eta}\left(\sum_{j\prec k}u_j^2\left(|s_j|+F_{j,n,\eta}\right)^2\right)^{1/2}, \qquad 3\leq k \leq n. 
  \end{equation}
\end{theorem}

We derive a closed-form error bound with the same techniques as in the proof of 
Theorem~\ref{thm:model2Analysis}.

\begin{theorem}\label{thm:mixPrecError}\label{t_52}
      Let $0<\eta < 1$, and $0 < \delta < 1-\eta$. Then under (\ref{model:second}), with probability at least $1-(\delta+\eta)$, the error in Algorithm \ref{alg:sum} is bounded by
  \begin{align*}
      |e_n| &\leq \sqrt{2\ln(2/\delta)}\left(1+\phi_{n,\tilde{h},\eta}\right)\sqrt{\sum_{k=2}^nu_k^2s_k^2} \\
            &\leq \sqrt{\tilde{h}}\sqrt{2\ln(2/\delta)}\left(1+\phi_{n,\tilde{h},\eta}\right)\sum_{k=1}^n|x_k|, 
  \end{align*}
  where $\tilde{h} \equiv \max_k \sum_{k\prec \ell \preceq n}u_\ell^2$ is the {\em weighted height} of the computational tree and 
 \begin{equation}\label{def_phi2}
    \phi_{n,\tilde{h},\eta} \equiv \lambda_{n,\eta}\sqrt{2\tilde{h}}\,u\,
    \exp\left(\lambda_{n,\eta}^2\tilde{h}u^2\right)\qquad 
\text{with}\qquad    
\lambda_{n,\eta} \equiv \sqrt{2\ln(2n/\eta)}.
\end{equation}
 \end{theorem}
 
\begin{proof}
Repeated application of the 2-norm triangle inequality implies that the bound
\begin{equation}\label{eqn:mixPrecError}
    |e_n| \leq \sqrt{2\ln(2/\delta)}\sum_{j=0}^h\lambda_{n,\eta}^j\left(\sum_{k=2}^nT_{k,j}^2u_k^2s_k^2\right)^{1/2},
\end{equation}
with
\begin{equation}
    T_{k,0}\equiv 1,\qquad T_{k,j} \equiv \left( \sum_{k\prec \ell_1\prec\cdots \prec \ell_j\preceq n}(u_{\ell_1}\cdots u_{\ell_n})^2 \right)^{1/2}, \qquad 2\leq k \leq n, 
\end{equation}
holds with  probability at least $1-(\delta+\eta)$.
Now apply the Cauchy-Schwarz inequality~\eqref{eqn:cauchySchwarz} as before and swap the order of summation, 
\begin{align}
\begin{split}
            \sum_{j=1}^h\lambda_{n,\eta}^j\left(\sum_{k=2}^nT_{k,j}^2u_k^2s_k^2\right)^{1/2} &\leq \left(\sum_{j=1}^h2^j\lambda_{n,\eta}^{2j}\sum_{k=2}^nT_{k,j}^2u_k^2s_k^2\right)^{1/2}\\
    &= \left(\sum_{k=2}^n \left(\sum_{j=1}^h 2^j\lambda_{n,\eta}^{2j}T_{k,j}^2\right)u_k^2s_k^2\right)^{1/2}.\label{eqn:mixPrecSum}
\end{split}
\end{align}
With $\tilde{h}_k \equiv \sum_{k\prec \ell \preceq n}u_\ell^2$ 
being the weighted depth of node $k$, the inner sums are bounded by
\begin{align*}
    \sum_{j=1}^h 2^j\lambda_{n,\eta}^{2j}T_{k,j}^2
    &= \prod_{k\prec \ell \preceq n}(1+2\lambda_{n,\eta}^2u_{\ell}^2) - 1, \qquad 2\leq k\leq n\\
        &\leq \exp\left(2\lambda_{n,\eta}^2\tilde{h}_k\right)-1
        \leq 2\lambda_{n,\eta}^2\tilde{h}_k\exp\left(2\lambda_{n,\eta}^2\tilde{h}_k\right),
\end{align*}
Insert the bounds $\tilde{h}_k\leq \tilde{h}$ into \eqref{eqn:mixPrecSum}, 
\begin{equation*}
    \sum_{j=1}^h\lambda_{n,\eta}^j\left(\sum_{k=2}^nT_{k,j}^2u_k^2s_k^2\right)^{1/2}
    \leq
    \lambda_{n,\eta}\sqrt{2\tilde{h}}\exp\left(\lambda_{n,\eta}^2\tilde{h}\right)\sqrt{\sum_{k=2}^nu_k^2s_k^2},
\end{equation*}
and combine this inequality with \eqref{eqn:mixPrecError}.
\end{proof}

As a corollary, we obtain the first rigorous probabilistic error bound for the mixed-precision version of \texttt{FABsum} \cite{blanchard2020class} in Algorithm~\ref{alg:fabsum}. 

\begin{algorithm}\caption{Mixed-precision \texttt{FABsum}}\label{alg:fabsum}
\begin{algorithmic}[1]
\REQUIRE Floating point numbers $x_1,\ldots,x_n$; block size $b$; precisions $u_{\text{lo}}$, $u_{\text{hi}}$
\ENSURE $s_n = \sum_{k=1}^n{x_k}$
\FOR{$k=1:n/b$}
\STATE $s_k$ = output of Algorithm~\ref{alg:sum} applied to $x_{(k-1)b+1},\ldots,x_{kb}$ in precision $u_{\text{lo}}$
\ENDFOR
\STATE $s_n$ = output of Algorithm~\ref{alg:sum} applied
to $s_1,\ldots,s_{n/b}$ in precision $u_{\text{hi}}$
\end{algorithmic}
\end{algorithm}

\begin{mycorollary}\label{c_fabsum}
Let $0<\eta < 1$; $0 < \delta < 1-\eta$;
 $h_{\text{lo}}$ the maximum height of all trees in the low-precision calls to Algorithm~\ref{alg:sum}; $h_{\text{hi}}$ the height of the portion of the tree in the high-precision call to Algorithm~\ref{alg:sum}.
Then under the mixed-precision extension of model (\ref{model:second}), with probability at least $1-(\delta+\eta)$, the error in Algorithm~\ref{alg:fabsum} is bounded by
       \[
        |e_n| \leq \sqrt{\tilde{h}}\sqrt{2\ln(2/\delta)}\left(1+\phi_{n,\tilde{h},\eta}\right)\sum_{k=1}^n|x_k|, 
       \]
       where $\tilde{h} \equiv h_{\text{lo}}u_{\text{lo}}^2 + h_{\text{hi}}u_{\text{hi}}^2$, and
       $\phi_{n,\tilde{h},\eta}$ is defined in \eqref{def_phi2}.
\end{mycorollary}

\begin{myremark}\label{r_51}
Inspired by the error expression in Theorem~\ref{thm:mixPrecError}, we offer the following modified version of advice in 
Remark~\ref{r_11}
\begin{quote}
\textit{In designing a mixed-precision summation method to achieve high accuracy, the aim should be to minimize the absolute values of the intermediate quantities $u_ks_k$.}
\end{quote}
The \texttt{FABsum} Algorithm~\ref{alg:fabsum} attempts to do just this by reserving its high-precision computations for the end, when the intermediate sums $s_k$ are likely to have larger magnitudes. 
\end{myremark}
\section{Numerical experiments}\label{s_numex}
After describing the setup, 
we present numerical experiments for recursive and pairwise summation (Section~\ref{s_ne1}), shifted summation (Section~\ref{s_ne2}), compensated summation (Section~\ref{s_ne3}),
and mixed-precision \texttt{FABSum} (Section~\ref{s_ne4}).

Experiments are performed in MATLAB R2022a, with unit roundoffs \cite{fp16}
\begin{itemize}
\item Half precision $u=2^{-11}\approx 4.88 \cdot 10^{-4}$. 
\item Single precision $u_{\text{hi}}=2^{-24}\approx 5.96\cdot 10^{-8}$ as the high precision in \texttt{FABsum} Algorithm~\ref{alg:fabsum}.
\item Double precision $u=2^{-53}\approx 1.11\cdot 10^{-16}$ for `exact' computation.
\end{itemize}
 Experiments plot errors from two rounding modes: round-to-nearest and stochastic rounding as implemented with  \texttt{chop} \cite{higham2019simulating}.

The summands $x_k$ are independent uniform $[0,1]$ random variables. The plots show relative errors $|\hat{s}_n-s_n|/|s_n|$ versus $n$,
for $100\leq n\leq 10^5$.
We choose relative errors rather than absolute errors
to allow for meaningful calibration:
Relative errors $\leq u$ indicate full accuracy; while
relative errors $\geq .5$ indicate zero digits of accuracy.

For shifted summation we use the empirical mean of two extreme summands,
\begin{equation*}
c=(\min_k{x_k}+\max_k{x_k})/2.
\end{equation*}
For probabilistic bounds, the combined failure probability is $\delta+\eta=10^{-2}+10^{-3}$, hence
$\sqrt{2\ln(2/\delta)}\approx 3.26$. For $n=10^5$ and $h=n$ we get $\lambda_{n,\eta}\approx 6.2$, and in half precision $u=2^{-11}$ the higher-order errors, $1+\phi_{n,h,\eta} \approx 4.4$, have a non-negligible effect on our bounds.  

\subsection{Sequential and pairwise summation}\label{s_ne1}
Figure \ref{f_ne1} shows the errors 
in half precision from Algorithm \ref{alg:sum} for
sequential summation in one panel, and for pairwise summation in another
panel, along with the deterministic bounds from Theorem \ref{t_gdet},
\begin{align}
|e_n| &\leq \sum_{k=2}^n|s_k||\delta_k|\prod_{k\prec j \preceq n}|1+\delta_j| 
\leq \lambda_h\,u\,\sum_{k=2}^n|s_k|\label{b_3}\\
&\leq \lambda_h\,h\,u\,\sum_{j=1}^n{|x_j|}.\label{b_4}
\end{align}
and the probabilistic bounds from Corollary~\ref{c_210},
\begin{align}
    |e_n| &\leq u\sqrt{2\ln(2/\delta)}\left(1+\phi_{n,h,\eta}\right)\sqrt{\sum_{k=2}^ns_k^2}\label{b_1}\\
    &\leq u \sqrt{h}\sqrt{2\ln(2/\delta)}\left(1+\phi_{n,h,\eta}\right)\sum_{k=1}^n|x_k|\label{b_2},
\end{align}

\paragraph{Sequential summation} 
The bound (\ref{b_1}) remains within a factor of 2
of (\ref{b_2}). Although the higher-order error terms $1+\phi_{n,h,\eta}$ represent only a small part of the error bounds, they may 
still be pessimistic, as the
bounds curve upwards for large $n$, while the actual errors 
increase more slowly. The reason may be the 
distribution of floating point numbers: spacing between consecutive numbers is constant within each interval $[2^t, 2^{t+1}]$, so a roundoff $\delta_k$ is affected by previous errors primarily if $\lfloor \log_2(\hat{s}_k)\rfloor \neq \lfloor \log_2(s_k)\rfloor$. Some analyses have derived deterministic error bounds for summation that do not contain second-order terms \cite{jeannerod2013improved,jeannerod2018relative,lange2019sharp,rump2012error}, and perhaps a more careful analysis will be able to do the same for probabilistic bounds. Our bounds otherwise accurately describe the behavior of stochastic rounding, but round-to-nearest suffers from stagnation for larger problem sizes. 

\paragraph{Pairwise summation} The bound (\ref{b_2}) grows proportional to $\sqrt{\log_2(n)}$,
while (\ref{b_1}) remains essentially constant. The behavior 
of (\ref{b_1}) may be due to the monotonically increasing partial sums for
uniform $[0, 1]$ inputs, where the final sum is likely to dominate all 
previous partial sums,
$(\sum_{k=2}^ns_k^2)^{1/2} = \mathcal{O}(s_n)$. This suggests
that pairwise summation of uniform $[0, 1]$ inputs is highly accurate.
The constant bound accurately describes the behavior of the error under stochastic rounding, but not round-to-nearest. We are not sure of the exact reason for the difference in behavior between the two. 

\begin{figure}
\begin{center}
\includegraphics[width = 0.49\textwidth]{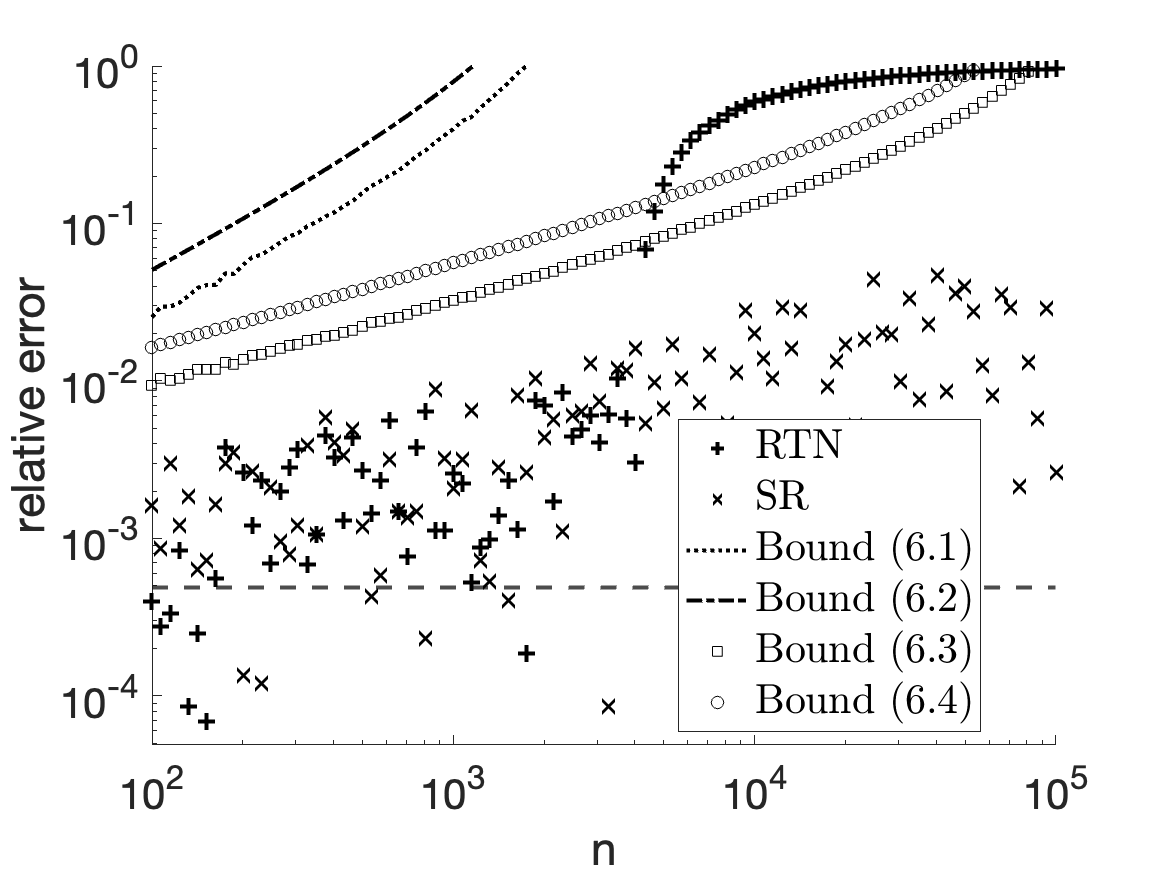}
\includegraphics[width = 0.49\textwidth]{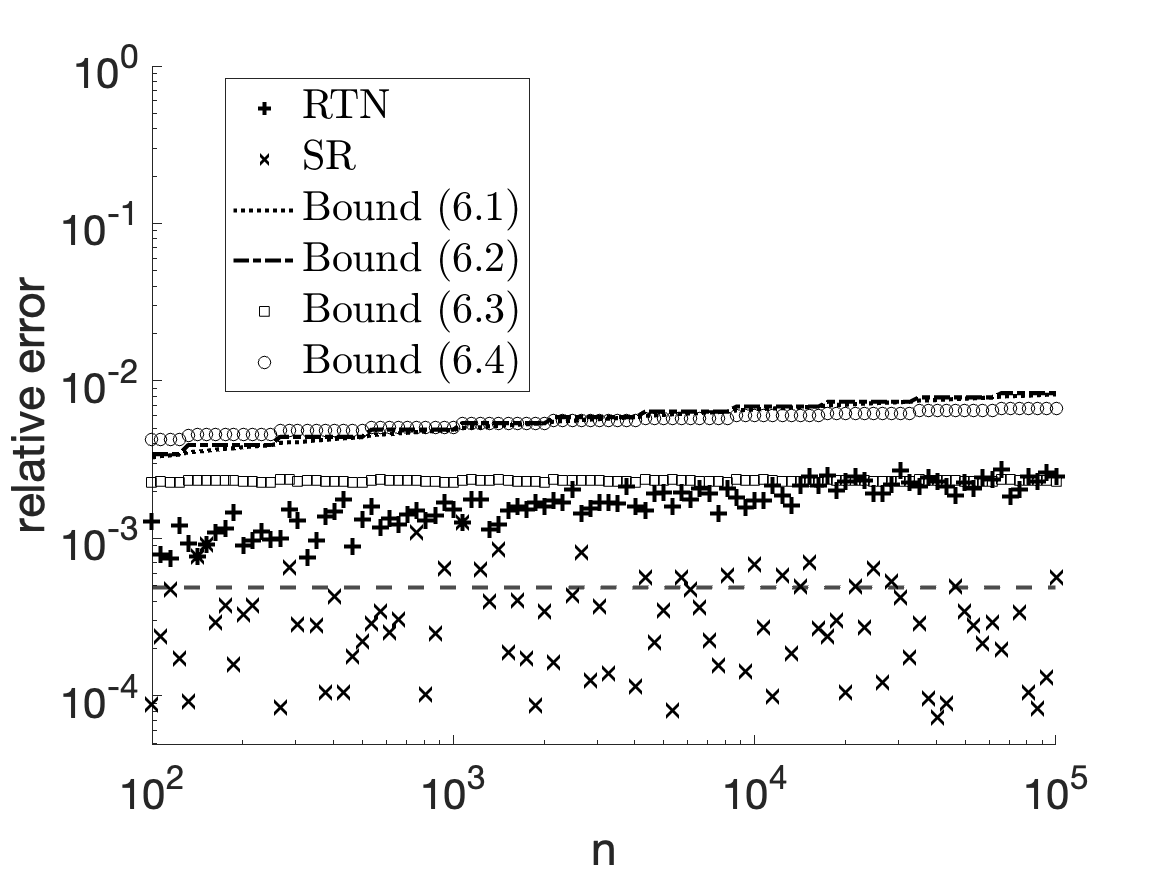}
\end{center}
\caption{Relative errors in half precision for recursive summation (left) and sequential summation (right) versus number of summands $n$.
The symbol (\texttt{+}) indicates round-to-nearest (RTN), and ($\mathtt{\times}$) indicates stochastic rounding (SR).
Horizontal line indicates unit roundoff $u = 2^{-11}$,
and remaining points indicate deterministic bounds (\ref{b_3}) and (\ref{b_4}) and probabilistic bounds (\ref{b_1}) and (\ref{b_2}).}
\label{f_ne1}
\end{figure}

\subsection{Shifted summation}\label{s_ne2}
Figure \ref{f_ne2} shows the errors 
in half precision from Algorithm \ref{alg:shiftSum} 
for shifted sequential summation and shifted pairwise summation,
along with the probabilistic bounds from Theorem~\ref{thm:shifted},
\begin{align}
    |e_n| &\leq u\sqrt{2\ln(2/\delta)}\left(1+\phi_{n,h,\eta}\right)\sqrt{s_n^2 + \sum_{k=2}^nt_k^2 + \sum_{k=1}^{n+1}y_k^2}\label{b_5}\\
    &\leq u\sqrt{2\ln(2/\delta)}\left(1+\phi_{n,h,\eta}\right)\left(n|c| +\sqrt{h}\sum_{k=1}^n{(|x_k-c|+|x_k|)}\right).\label{b_6}
\end{align}
 
A comparison with Figure \ref{f_ne1}
shows that shifting reduces both the actual errors and the bounds.
Errors are on the order of unit roundoff, in all cases: 
round-to-nearest and stochastic rounding, and sequential and pairwise summation.

\begin{figure}
\begin{center}
\includegraphics[width = 0.49\textwidth]{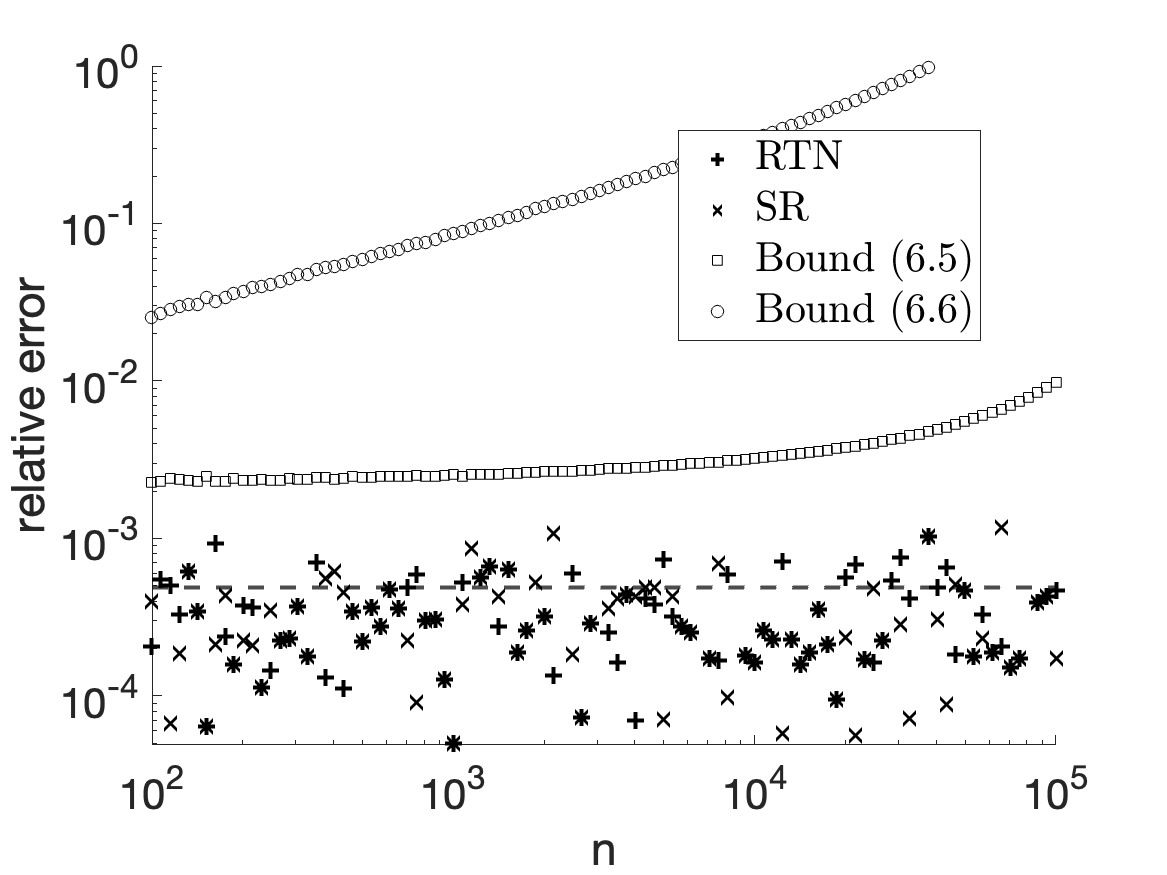}
\includegraphics[width = 0.49\textwidth]{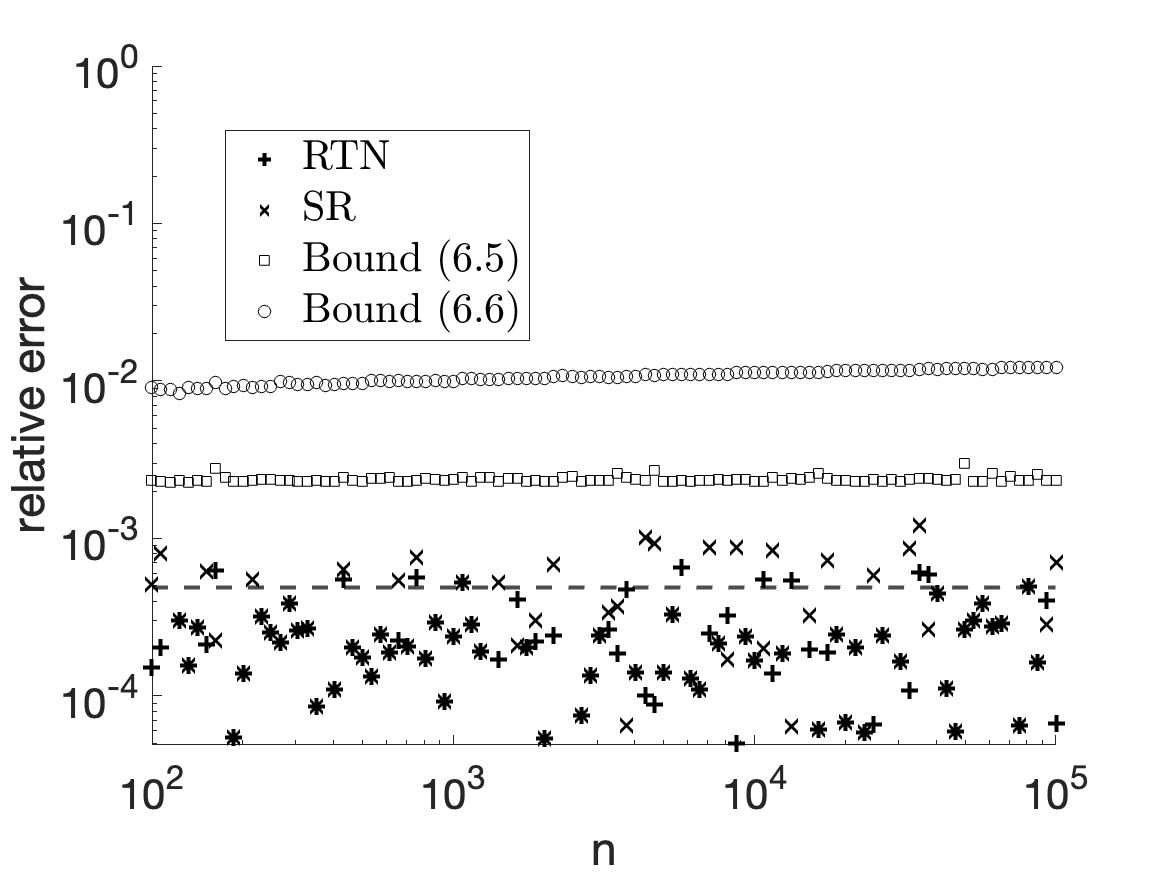}
\end{center}
\caption{Relative errors in half precision for shifted sequential summation (left) and shifted pairwise summation (right) versus number of summands $n$. 
The symbol (\texttt{+}) indicates round-to-nearest (RTN), and ($\mathtt{\times}$) indicates stochastic rounding (SR).
Horizontal line indicates unit roundoff $u = 2^{-11}$,
and remaining points indicate  probabilistic bounds (\ref{b_5}) and (\ref{b_6}).}
\label{f_ne2}
\end{figure}

\subsection{Compensated summation}\label{s_ne3}
The first panel in Figure~\ref{f_ne3} shows the errors in half precision for Algorithm~\ref{alg:compensated}
for $10^2\leq n \leq 10^7$ summands\footnote{Our simulation of half-precision ignores the range restriction \texttt{realmax} = 65504.},
along with deterministic bounds derived from Corollary~\ref{c_cs5}, 
\begin{align}
    |e_n| &\leq u|s_n| + 2u(1+3u)\sum_{k=2}^n|x_k| + 4u^2\sum_{k=2}^{n-1}|s_k| + \mathcal{O}(u^3) \label{b_9}\\
    &\leq (3u + (4n-2)u^2)\sum_{k=1}^n|x_k| + \mathcal{O}(u^3), \label{b_10}
\end{align}
and the probabilistic bounds from Theorem~\ref{thm:comp_prob_err},
\begin{align}
         |e_n|&\leq u\sqrt{2\ln(2/\delta)}
\left(|s_n| + \gamma(\sqrt{2}+\alpha u)\sqrt{\sum_{k=2}^nx_k^2}+\gamma\alpha u\sqrt{\sum_{k=2}^ns_k^2}\right) \label{b_7}\\
&\leq u\sqrt{2\ln(2/\delta)}
\left(1+\sqrt{2}+\sqrt{6}(\sqrt{n}+1) u\right)\sum_{k=1}^n|x_k| + \mathcal{O}(u^3). \label{b_8}
\end{align}

The probabilistic bounds (\ref{b_7}) and
(\ref{b_8}) track the error behavior accurately, with (\ref{b_7}) even capturing
the correct order of magnitude. This also
illustrates the higher accuracy of bounds
involving partial sums.

\begin{figure}
\begin{center}
\includegraphics[width = 0.49\textwidth]{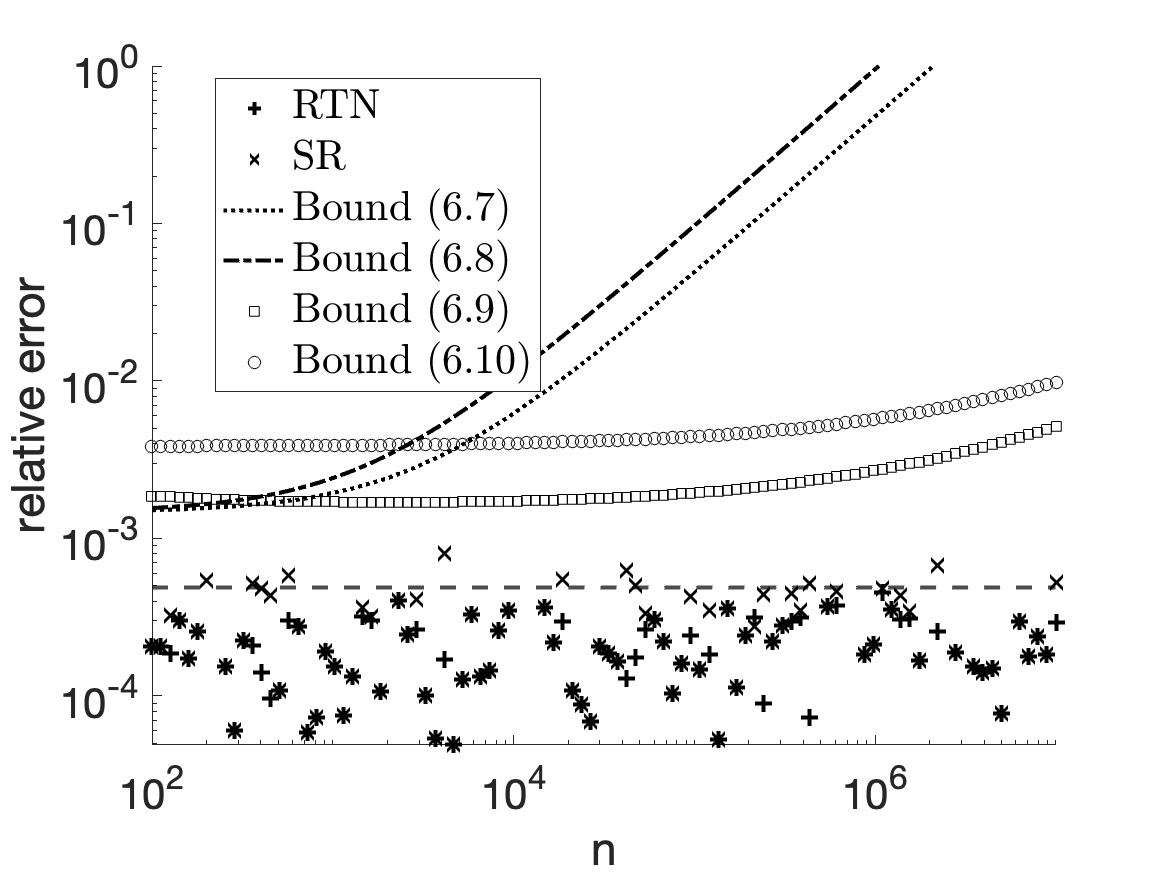}
\includegraphics[width = 0.49\textwidth]{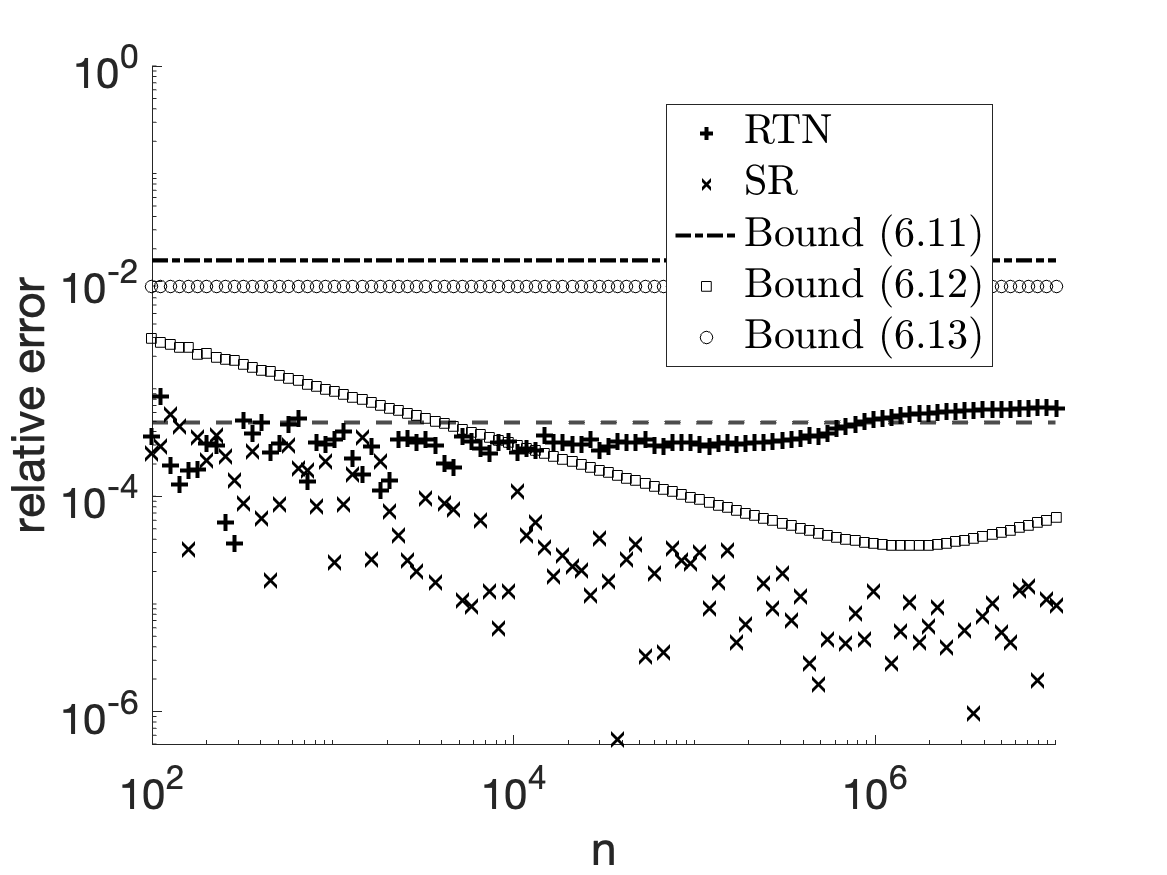}
\end{center}
\caption{Relative errors in half precision for compensated summation (left) and mixed
precision with 
\texttt{FABsum} with high precision $u_{\text{hi}}= 2^{-24}$ (right) versus number of summands $n$. The symbol (\texttt{+}) indicates round-to-nearest (RTN), and ($\mathtt{\times}$) indicates stochastic rounding (SR).
Horizontal line indicates unit roundoff $u_{\text{lo}} = 2^{-11}$,
and remaining points indicate bounds \eqref{b_9}-\eqref{b_8} (left) and \eqref{b_11}-\eqref{b_13} (right). 
}
\label{f_ne3}
\end{figure}

\subsection{Mixed-precision \texttt{FABsum} summation}\label{s_ne4} The second panel of Figure~\ref{f_ne3} shows the errors for Algorithm~\ref{alg:fabsum} 
with $u_{\text{lo}}=2^{-11}\approx 4.44\cdot 10^{-4}$, $u_{\texttt{hi}}=2^{-24}\approx 5.96\cdot 10^{-8}$, block 
size $b=32$ and $10^2\leq n \leq 10^7$ summands, where each internal call to Algorithm~\ref{alg:sum} uses recursive summation. We also plot the deterministic first-order bound from
\cite[Eqn.~3.5]{blanchard2020class}, 

\begin{equation}
    |e_n| \leq bu\sum_{k=1}^n|x_k| + \mathcal{O}(u^2), \label{b_11}
\end{equation}
and the probabilistic bounds derived from Theorem~\ref{t_52}, 
 \begin{align}
     |e_n| &\leq \sqrt{2\ln(2/\delta)}\left(1+\phi_{n,\tilde{h},\eta}\right)\sqrt{\sum_{k=2}^nu_k^2s_k^2} \label{b_12}\\
    &\leq \sqrt{\tilde{h}}\sqrt{2\ln(2/\delta)}\left(1+\phi_{n,\tilde{h},\eta}\right)\sum_{k=1}^n|x_k|, \label{b_13}
 \end{align}
 where $\tilde{h} = bu^2 + (n/b)u_{\text{hi}}^2$. Errors are on the order of unit roundoff for round-to-nearest. We were surprised to observe that for stochastic rounding, errors fell to more than an order of magnitude {\em below} unit roundoff for large problem sizes. This behavior is correctly predicted by the bound in terms of the partial sums \eqref{b_12} but not the bound in terms of the inputs \eqref{b_13}, demonstrating the importance of error expressions involving the 
 partial sums.

\subsection*{Acknowledgement}
We are greatly indebted to Claude-Pierre Jeannerod for his many, many 
suggestions that improved the paper.
We also thank Johnathan Rhyne for helpful discussions.
\appendix


\section{Proof of Lemma \ref{lemma:alphaBound}}\label{apx:alpha}
Define $\beta \equiv u(1+u)^2$ and \begin{equation}
    \omega_k \equiv |s_k| + |x_k| + S_k, \qquad 2\leq k \leq n-1.
\end{equation}
Lemma \ref{lemma:sddot} implies
\begin{align}
    Z_k &= u\omega_k +(1+u)Y_k = u\omega_k + (1+u)^2C_{k-1},\qquad 3\leq k\leq n-1 \label{eqn:apx_zrec}\\
    C_k &= u\omega_k + uZ_k
        = u(1+u)\omega_k + \beta C_{k-1},\label{eqn:apx_crec}
\end{align}
where $Z_2\leq u\omega_2$ and $C_2\leq u(1+u)\omega_2$. For $3\leq k \leq n$, 
define the vectors
\[
{\bf c}_k \equiv \begin{bmatrix}C_{k-1}&\cdots& C_2\end{bmatrix}^T, \quad {\bf z}_k \equiv \begin{bmatrix}Z_{k-1}&\ldots&Z_2\end{bmatrix}^T, \quad {\bf w}_k \equiv \begin{bmatrix}\omega_{k-1} &\ldots &
\omega_2\end{bmatrix}^T.
\]
From \eqref{eqn:apx_crec} follows the
componentwise inequality
\[
    {\bf c}_k \leq u(1+u){\bf w}_k + \beta {\bf U}{\bf c}_k, 
\]
where ${\bf U}$ is an upper shift matrix.
Solving for ${\bf c}_k$ gives another componentwise inequality with a unit upper
triangular matrix ${\bf I}-\beta{\bf U}$,
\[
    {\bf c}_k \leq u(1+u)({\bf I}-\beta{\bf U})^{-1}{\bf w}_k,
\]
and a bound
\[
    \|{\bf c}_k\|_2 \leq u(1+u)\|({\bf I}-\beta{\bf U})^{-1}{\bf w}_k\|_2 \leq \tfrac{u(1+u)}{1-\beta}\|{\bf w}_k\|_2. 
\]
The bound for $\|{\bf z}_k\|_2$ follows from \eqref{eqn:apx_zrec}
and the definition of $\beta$,
\[
    \|{\bf z}_k\|_2 \leq u\|{\bf w}_k\|_2 + (1+u)^2\|{\bf c}_k\|_2 \leq \tfrac{u(2+2u+u^2)}{1-\beta}\|{\bf w}_k\|_2.
\]
Finally, from $Y_k = (1+u)C_{k-1}$ follows the Frobenius norm bound
\begin{equation*}
    \left(\sum_{j=3}^k\left( Y_j^2 + C_{j-1}^2 + Z_{j-1}^2\right)\right)^{1/2}
    = \left\|\begin{bmatrix}(1+u){\bf c}_k& {\bf c}_k& {\bf z}_k\end{bmatrix}\right\|_F \leq \alpha u\|{\bf w}_k\|_2, 
\end{equation*}
where the higher order terms in $\alpha$ follow from the
Taylor series expansion $(1-\beta)^{-2}=1 +2u +\mathcal{O}(u^2)$,
\[
    \alpha^2 = \frac{1+3(1+u)^2 + 2(1+u)^4}{(1-\beta)^2} = 6 + 26u + \mathcal{O}(u^2). 
\]

\bibliographystyle{spmpsci}
\bibliography{references}

\end{document}